\RequirePackage[l2tabu,orthodox,abort]{nag}

\documentclass{birkjour}

\usepackage{fixltx2e}
\usepackage[all,error]{onlyamsmath}

\usepackage{amssymb,amsfonts,amsmath,
amsrefs,enumerate,color,url,hyperref,mathbbol,dsfont}
\usepackage{mathrsfs}
\usepackage{tikz}

\usepackage{etoolbox}

\newtoggle{blind}

\togglefalse{blind} 

\usepackage[T1]{fontenc}
\usepackage[cp1250]{inputenc}

\usepackage[strict=true]{csquotes}
 \newtheorem{thm}{Theorem}[section]
 
 \newtheorem{lem}[thm]{Lemma}
 \newtheorem{prop}[thm]{Proposition}
 \theoremstyle{definition}
 
 \theoremstyle{remark}

 \numberwithin{equation}{section}
\usepackage{mathtools}

\newcommand{\ddwad}{\Delta_{2D}}
\newcommand{\ddwadd}{\overline{\ddwad}}
\newcommand{\notzee}{C_{\not z}(\Omega)}

\usepackage{enumitem}
\renewcommand{\theenumi}{\textup{(\alph{enumi})}}
\renewcommand{\labelenumi}{\theenumi}

\begin{document}

			
\title[Semigroups and thin layers]{Semigroup-theoretic approach to diffusion in thin layers separated by semi-permeable membranes}


\iftoggle{blind}{\author{ }}{
\author[A. Bobrowski]{Adam Bobrowski}
\address{
Lublin University of Technology\\
Nadbystrzycka 38A\\
20-618 Lublin, Poland}

\email{a.bobrowski@pollub.pl}}





\newcommand{\cxi}{(\xi_i)_{i\in \N} }
\newcommand{\lam}{\lambda}
\newcommand{\eps}{\varepsilon}
\newcommand{\ud}{\, \mathrm{d}}
\newcommand{\udd}{\mathrm{d}}
\newcommand{\pr}{\mathbb{P}}
\newcommand{\f}{\mathcal{F}}
\newcommand{\h}{\mathcal{H}}
\newcommand{\ai}{\mathcal{I}}
\newcommand{\R}{\mathbb{R}}
\newcommand{\C}{\mathbb{C}}
\newcommand{\Z}{\mathbb{Z}}
\newcommand{\N}{\mathbb{N}}
\newcommand{\Y}{\mathbb{Y}}
\newcommand{\e}{\mathrm {e}}
\newcommand{\tif}{\widetilde {f}}
\newcommand{\slam}{\sqrt {\lam}}
\newcommand{\Id}{{\mathrm{Id}}}
\newcommand{\cic}{C_{\mathrm{mp}}}
\newcommand{\cod}{C_{\mathrm{odd}}[0,1]}
\newcommand{\cev}{C_{\mathrm{even}}[0,1]}
\newcommand{\cevr}{C_{\mathrm{even}}(\mathbb{R})}
\newcommand{\codr}{C_{\mathrm{odd}}(\mathbb{R})}
\newcommand{\cez}{C_0(0,1]}
\newcommand{\fod}{f_{\mathrm{odd}}} 
\newcommand{\fev}{f_{\mathrm{even}}} 
\newcommand{\sem}[1]{\mbox{$\left (\e^{t{#1}}\right )_{t \ge 0}$}}
\newcommand{\semi}[1]{\mbox{$\left ({#1}\right )_{t > 0}$}}
\newcommand{\semt}[2]{\mbox{$\left (\e^{t{#1}} \otimes_\varepsilon \e^{t{#2}} \right )_{t \ge 0}$}}
\newcommand{\tr}{\textcolor{red}}
\newcommand{\cea}{C_A}
\newcommand{\ceat}{C_A(t)}
\newcommand{\cosinea}{(\ceat )_{t\in \R}}  
\newcommand{\sea}{S_A}
\newcommand{\seat}{S_A(t)}
\newcommand{\sema}{(\seat )_{t\ge 0}}
\newcommand{\wt}{\widetilde}
\renewcommand{\iff}{if and only if }
\renewcommand{\k}{\mathrm{k}}
\newcommand{\tcm}{\textcolor{magenta}}
\newcommand{\tcb}{\textcolor{blue}}
\newcommand{\dx}{\ \textrm {d} x}
\newcommand{\dy}{\ \textrm {d} y}
\newcommand{\dz}{\ \textrm {d} z}
\newcommand{\di}{\textrm{d}}
\newcommand{\tcg}{\textcolor{green}}
\newcommand{\lc}{\mathfrak L_c}
\newcommand{\ls}{\mathfrak L_s}
\newcommand{\grat}{\lim_{t\to \infty}}
\newcommand{\gra}{\lim_{n\to \infty}}
\newcommand{\grae}{\lim_{\eps \to 0}}
\newcommand{\rez}[1]{\left (\lam - #1\right)^{-1}}
\newcommand{\papa}{\hfill $\square$}
\newcommand{\papap}{\end{proof}}
\newcommand {\x}{\mathbb{X}}
\newcommand{\aex}{A_{\mathrm ex}}
\newcommand{\jcg}[1]{\left ( #1 \right )_{n\ge 1} }
\newcommand {\y}{\mathbb{Y}}
\newcommand{\injtp}{\x \hat \otimes_{\varepsilon} \y}
\newcommand{\pin}{\|_{\varepsilon}}
\newcommand{\mc}{\mathcal}
\newcommand{\inter}{\left [0, 1\right ]}
\newcommand{\lir}{\lim_{r \to 1}}
\newcommand{\ha}{\mathfrak {H}}
\newcommand{\dom}[1]{D(#1)}
\newcommand{\mquad}[1]{\quad\text{#1}\quad}
\makeatletter
\newcommand{\normt}{\@ifstar\@normts\@normt}
\newcommand{\@normts}[1]{%
  \left|\mkern-1.5mu\left|\mkern-1.5mu\left|
   #1
  \right|\mkern-1.5mu\right|\mkern-1.5mu\right|
}
\newcommand{\@normt}[2][]{%
  \mathopen{#1|\mkern-1.5mu#1|\mkern-1.5mu#1|}
  #2
  \mathclose{#1|\mkern-1.5mu#1|\mkern-1.5mu#1|}
}
\makeatother

\thanks{Version of \today}

 \iftoggle{blind}{
   \begin{titlepage}
   \thispagestyle{empty}
   
   \noindent {\Large\textbf{Semigroup-theoretic approach to diffusion in thin
       layers separated by semi-permeable membranes}}

   \vspace{0.5cm}

   \noindent Adam Bobrowski (\texttt{a.bobrowski@pollub.pl})

   \vspace{0.5cm}
   \noindent Lublin University of Technology\\
   Nadbystrzycka 38A\\
   20-618 Lublin, Poland

   \vspace{0.5cm}

   \noindent \textbf {Acknowledgment.}
   This research is supported by National Science Center (Poland) grant
   2017/25/B/ST1/01804.
   \end{titlepage}
}

\begin{abstract}Using techniques of the theory of semigroups of linear operators we study the question of approximating solutions to equations governing diffusion in thin layers separated by a semi-permeable membrane. We show that as thickness of the layers converges to $0$, the solutions, which by nature are functions of $3$ variables, gradually lose dependence on the vertical variable and thus may be regarded as functions of $2$ variables. The limit equation describes diffusion on 
the lower and upper sides of a two-dimensional surface (the membrane) with jumps from one side to the other. The latter possibility  
is expressed as an additional term in the generator of the limit semigroup, and this term is build from permeability coefficients of the membrane featuring in the transmission conditions of the approximating equations (i.e., in the description of the domains of the generators of the approximating semigroups). We prove this convergence result in the spaces of square integrable and continuous functions, and study the way the choice of transmission conditions influences the limit.

\end{abstract}

\iftoggle{blind}{\thispagestyle{firstpage}} %

\maketitle

\newcommand{\oper}{\mathfrak R_r}
\newcommand{\opern}{\mathfrak R_{\rn}^\mho}
\newcommand{\brn}{\mbox{$\Delta^\mho_{\rn}$}}
\newcommand{\bro}{\mbox{$\Delta_{\rn}$}}
\newcommand{\rn}{r}
\newcommand{\cern}{C\hspace{-0.07cm}\left[\rn, 1\right ]}
\newcommand{\cernbez}{C\left[\rn, 1\right ]}
\newcommand{\cep}{C\hspace{-0.07cm}\left[ 0, 1\right ]}
\newcommand{\copi}{C[0,\pi]}
\newcommand{\cerec}{C\hspace{-0.07cm} \left ([0,\pi]\times [r,1]\right)}
\newcommand{\cerecbez}{C \left ([0,\pi]\times [r,1]\right)}
\newcommand{\cerecdwa}{C^2\hspace{-0.07cm} \left ([0,\pi]\times [r,1]\right)}
\newcommand{\cerecj}{C\hspace{-0.07cm} \left ([0,\pi]\times \left [0 ,1\right ]\right)}
\newcommand{\xprim}{C_\theta (UR)}
\newcommand{\ie}{i.e., }
\newcommand{\rla}{R_\lambda}
\newcommand{\grubex}{\mathbb X}
\newcommand{\Jcg}[1]{\left ( #1 \right )_{i=1,...,N}} 

\section{Introduction}\label{intro}

The paper is devoted to a semigroup-theoretical approach to the problem of approximating solutions to an equation modeling diffusion in two thin 3D layers separated by a semi-permeable membrane: particles diffusing in the upper layer may, via a stochastic mechanism, filter through the membrane to the lower layer and continue  their chaotic movement there, and vice versa. To this end,  the reaction-diffusion equation
\begin{equation}\label{intro:1} \partial_t u = \Delta_{3D} u + F(u) \end{equation}
where $\Delta_{3D}$ is a $3D$ Laplace operator, and $F$ is a Lipschitz continuous forcing (reaction) term, considered in two layers of thickness $\eps $, is equipped with boundary and transmission conditions (see \eqref{eq.robin2} and \eqref{C:eq.robin2}, further down) describing in particular the way the membrane works, and appropriate generation theorems are proved in the spaces of square integrable and continuous functions, respectively (thus establishing existence and uniqueness of mild solutions of the equation). Next, we show that, as $\eps \to 0$, the approximating processes resemble more and more 2D Brownian motions on the upper and lower sides of the membrane. Remarkably, the limit process allows also jumps from one side to the other: this possibility is the limit equivalent of the mechanism of filtering through the membrane in the approximating process. More specifically, as $\eps \to 0$ and as looked upon through a magnifying glass (see below), solutions of \eqref{intro:1} become more and more flat in the vertical direction (but still differ in the lower and upper layers) and thus may be identified with pairs $(u^-,u^+)$ of functions of two variables, defined on the lower and upper sides of the membrane. The limit dynamics is then governed by the following system: 
\begin{align}
\partial_t \binom {u^-}{u^+}  &= \left [ \begin{pmatrix} \Delta_{2D}  & 0 \\ 0 & \Delta_{2D}\end{pmatrix}+ \begin{pmatrix} - \alpha & \alpha \\ \beta & - \beta \end{pmatrix}\right ] \binom {u^-}{u^+}  + \binom{F(u^-)- c^-u^-}{F(u^+)-c^+u^+  }\label{intro:2} 
\end{align}
where $\Delta_{2D}$ is a 2D Laplace operator. More interestingly, $\alpha $ and $\beta$ are functions describing permeability of the membrane in the approximating processes. Thus, our theorem says that in the limit, transmission conditions governing the approximating processes become integral parts of the master equation. Functions $c^-$ and $c^+$ play a similar role: they come from the Robin boundary conditions in the approximating processes, describing partial loss of particles touching lower and upper boundaries of the layers.    It is worth noting that processes described by \eqref{intro:2} are closely related to piecewise deterministic Markov processes of M.H.A Davis \cites{davis,davisk,davisk2,rudnickityran}, random evolutions of R.J. Griego and R. Hersh \cites{ethier,gh1,gh2,pinskyrandom} and to randomly switching diffusions \cites{dwoje,ilin,yin}; for a semigroup theoretic context see  \cites{convex,knigaz}.

We prove two variants of generation and approximation theorems. In Section \ref{ldwa}, devoted to analysis in $L^2$, permeability coefficients are bounded, measurable functions on the membrane, and thus permeability may vary from region to region. In Section \ref{aic} where we construct Feller semigroups,  we restrict ourselves to constant permeability coefficients but, on the other hand, show that our approximation theorem is robust to changing the mechanism of filtering through the membrane: its thesis remains (almost) the same even if more general transmission conditions than those considered in Section \ref{ldwa} are considered.

It should be noted here that, beginning with the seminal paper \cite{hale}, considerable attention has already been given to thin layer approximation, both in the mathematical and in the physical literature; see e.g. \cites{arrieta,barros,raugel,prizziryba,prizziryba2,prizzi,elsken} for the former and \cite{carlsson} for an example of the latter. Our paper, however, differs from the previous works in several aspects. First of all, we look at the problem of thin layer approximation from the perspective of convergence of semigroups of linear operators, and use the tools of this rich theory (see e.g., \cite{knigaz}). Secondly, in Section \ref{aic}, we show that the theory may be successfully applied to Feller semigroups acting in the spaces of continuous functions. By contrast, a usual machinery used in the literature is the method of forms, and the analysis is usually carried out in the space of square integrable functions. This change of perspective is important for at least three reasons: (i) Convergence of Feller semigroups is, by the Trotter--Sova--Kurtz--Mackevi\v cius Theorem (see e.g. \cite{kallenbergnew} p. 385), equivalent to weak convergence of the Markov processes involved, whereas a stochastic interpretation of a similar convergence result in a Hilbert space is rather unclear.  (ii) The analysis of the thin layer approximation hinges on a stretching transformation of `thin coordinates' (see our Section \ref{avtamg}, compare p. 111 in \cite{prizzi} or p. 583 in \cite{pazanin}). This transformation is an isometric isomorphism of appropriate spaces of continuous functions, but not of the related Hilbert spaces. (iii) Uniform convergence, i.e., convergence in the norm of the space of continuous functions, is stronger than that in the norm of $L^2$ (since the region we consider is bounded). 

On the other hand, as exemplified also by the results of our Section \ref{ldwa}, analytic tools available in $L^2$ are more flexible and thus allow treating potentially more general geometries and more general boundary conditions. Perhaps stochastic analysis, and the results presented in \cite{kurtzcontrol} and \cite{kurtzcostantini} in particular, may lead to generalizations that are also meaningful for stochastic processes.

The last and probably most significant difference between this paper and the existing literature is that, while our theorem is focused on the intriguing fact that boundary and transmission conditions in the limit become integral parts of the master equation, the previous papers are generally devoted to Neumann boundary conditions, which do not lead to such a phenomenon. An exception to this rule is the recent paper \cite{pazanin} which involves Robin-type boundary conditions; again, the analysis in that paper is carried out in Sobolev-type Hilbert space and allows treating more general geometries and boundary conditions than in our present paper, but the question of the role of boundary conditions in the limit and in the approximating equations is not discussed there. (Moreover, \cite{pazanin} is devoted to a completely different equation.)  There are also remarks on Robin-type boundary conditions in \cite{hale} and \cite{raugel}, but they are of marginal character: to the best of our knowledge, the fact that in the thin layer approximation boundary and transmission conditions `jump into' the limit equation has not been appropriately described yet.

Finally, it should be mentioned that, while the present paper is purely theoretical, the original inspirations to this analysis came from modeling diffusion of kinases in $B$ lymphocytes which have extremely large nuclei and thus the 3D region where kinases diffuse resembles a 2D manifold. See  
\cites{hat,hat2011,kazlip,dlajima,thin}.

\newcommand{\be}{\mc B}

\section{Analysis in $L^2$}\label{ldwa}

\subsection{Intuitions and the limit master equation} \label{iatlme}

For our case-study we choose the following situation. Let $\be $ 
(`b' for `base') be a bounded, open  subset of $\R^2$ with Lipschitz boundary. Given $\eps \in (0,1]$, we consider the following `split' solid (cylinder) 
$\Omega_\eps = \Omega_\eps^- \cup \Omega_\eps^+$ formed by two bordering thin layers 
\begin{align*}
 \Omega_\eps^- &\coloneqq \{ (x,y,z)\in \R^3 : (x,y) \in \be, -\eps < z < 0\},\\
 \Omega_\eps^+ &\coloneqq \{ (x,y,z)\in \R^3 : (x,y) \in \be, 0 < z <\eps\}. \\
\end{align*}
This solid's  lower and upper bases will be denoted 
\newcommand{\bue}{\be_\eps^{\text{up}}}
\newcommand{\ble}{\be_\eps^{\text{lo}}}
\newcommand{\bpe}{\be^{+}}
\newcommand{\bme}{\be^-}
\newcommand{\buej}{\be_1^{\text{up}}}
\newcommand{\blej}{\be_1^{\text{lo}}}
\newcommand{\bpej}{\be_1^{+}}
\newcommand{\bmej}{\be_1-}
\newcommand{\buen}{\be^{\text{up}}}
\newcommand{\blen}{\be^{\text{lo}}}
\newcommand{\bpen}{\be^{+}}
\newcommand{\bmen}{\be^-}
\begin{align*}
 \ble &\coloneqq \{ (x,y,z)\in \R^3 : (x,y) \in \be, z=-\eps \},\\
 \bue &\coloneqq \{ (x,y,z)\in \R^3 : (x,y) \in \be, z =\eps\}.
\end{align*}
We imagine that the plane separating the layers $\Omega_\eps^-$ and $\Omega_\eps^+$ is a semi-permeable membrane, and thus distinguish between what is `right above' and `right below' this plane. Therefore, we introduce (see Figure \ref{ef1}): 
\begin{align*}
\bme &\coloneqq \{ (x,y,z)\in \R^3 : (x,y) \in B, z=0-\},\\ 
\bpe &\coloneqq \{ (x,y,z)\in \R^3 : (x,y) \in B, z =0+\}.
\end{align*}
\begin{figure}
\includegraphics[scale=0.6]{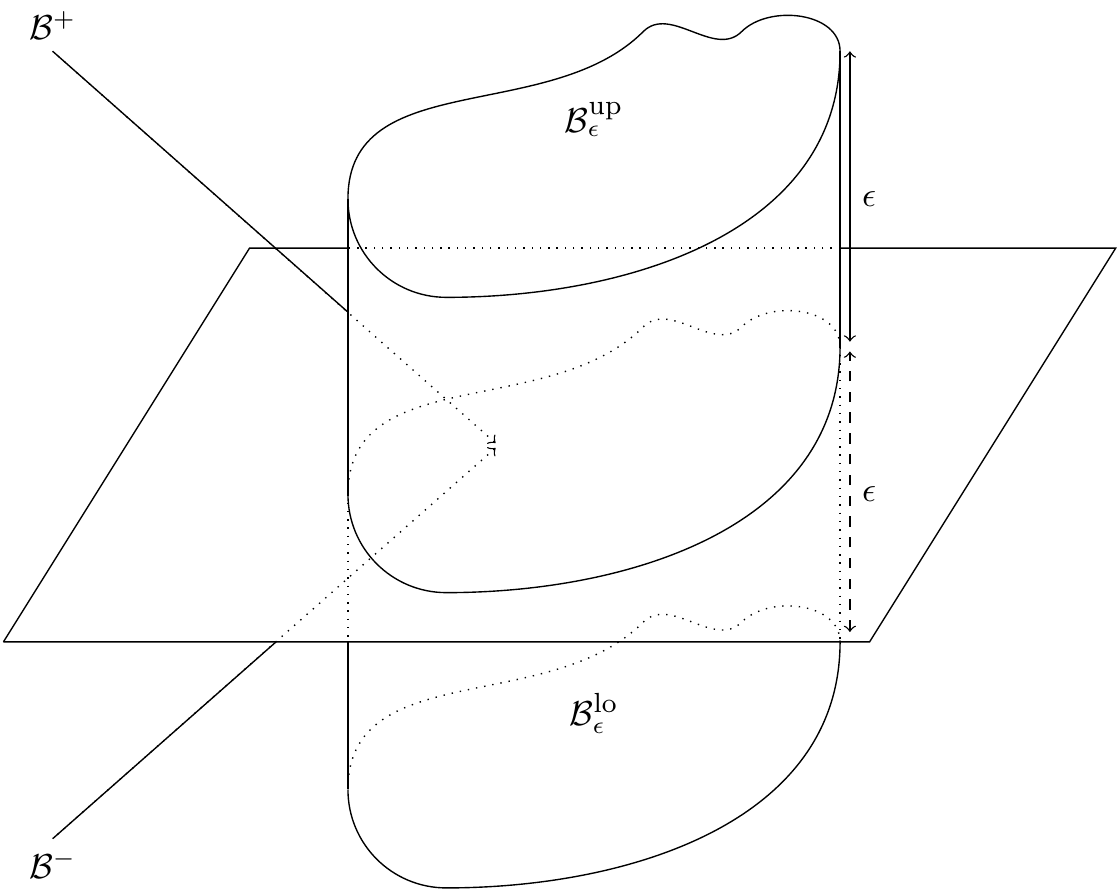}
\caption{Two thin layers separated by a semi-permeable flat membrane}\label{ef1}
\end{figure}
\hspace{-0.2cm} Having prepared the stage for the analysis, we consider the following reaction-diffusion equation in \( \Omega_\eps \): \begin{align}
\label{eq.diff1}
\partial_t u(t,x,y,z) &= \Delta_{3D} u(t,x,y,z)+F (u(t,x,y,z)), \, (x,y,z) \in  \Omega_\eps, t >0, \end{align}
where $\Delta_{3D}=\partial_x^2 + \partial_y^2 + \partial_z^2,$
and $F:\R \to \R$ in the reaction term is assumed globally Lipschitz continuous. 
In addition to usual Neumann boundary conditions on the sides of the solid, on the bases $\bue$ and $\ble$ we impose Robin boundary conditions of the form
\begin{align}  
\partial_z u(t,x,y,-\eps )&= \eps c^-(x,y)u(t,x,y,-\eps), 
   \label{eq.robin1} \\ 
\partial_z u(t,x,y,\eps)&=  -\eps c^+(x,y)u(t,x,y,\eps),\qquad (x,y) \in \be, t >0, \nonumber
\end{align}
where $c^-,c^+: \R^2 \to \R$ 
are given, measurable, essentially bounded functions. As explained in \cites{dlajima,thin}, the scaling factor (i.e., $\eps$) is needed in these boundary conditions; otherwise the limit (as $\eps\to 0$) discussed below will be uninteresting. These boundary conditions describe a stochastic mechanism of removing some of the diffusing particles touching the boundaries in regions where $c^-$ and $c^+$ are positive, and a similar mechanism of an inflow of new particles from regions of the boundary where $c^-$ and $c^+$ are negative.  

Moreover, on $\bme$ and $\bpe$ we impose the following transmission conditions modeling semi-permeability of the membrane: 
\begin{align}
\partial_z u(t,x,y,0-)&= \eps \alpha (x,y)[u(x,y,0-)-u(x,y,0+)],
   \label{eq.robin2} \\ 
\partial_z u(t,x,y,0+)&= \eps \beta (x,y)[u(x,y,0+)-u(x,y,0-)], \qquad (x,y) \in \be, t >0, \nonumber
\end{align}
where $\alpha, \beta: \R^2 \to [0,\infty)$ are given, essentially bounded functions. These conditions describe a stochastic mechanism in which a particle in the upper layer may filter through the membrane to the lower layer, and vice versa. The functions $\alpha$ and $\beta$ are permeability coefficients: the larger is $\alpha$ in a given subset of the membrane, the shorter is the average time needed for a particle to filter through this part of the membrane from the lower to the upper layer, and the larger is the $\beta$ the shorter is the average time for a particle to filter from the upper to the lower layer (see \cites{bobmor,zmarkusem,lejayn}, see also \cite{knigaz}*{p. 66} and the references given there).  
For our subsequent analysis (i.e., our generation and convergence results) the assumption that 
$\alpha$ and $\beta$ are non-negative is not needed; but a probabilistic interpretation of transmission conditions \eqref{eq.robin2} with negative $\alpha$ and $\beta$ is less pleasing.

Our main objective is to study behavior of solutions to equations \eqref{eq.diff1}--\eqref{eq.robin2} as $\eps$ converges to $0$.
We will argue that in this process, these solutions become more and more `flat in the $z$-direction', i.e., become less and less dependent on $z$, and thus in the limit may be regarded as functions of two variables (plus time). The functions obtained in the lower and upper layers differ, however, and thus will be denoted $u^-(t,x,y)$ and $u^+(t,x,y)$, respectively, and interpreted as functions on the upper and lower sides of the membrane.  As it transpires, dynamics  of $u^-$ and $u^+$ in time is governed by the master equation \eqref{intro:2}, where $F(u^-)$ is a shorthand for the function $(x,y)\mapsto F(u^-(x,y))$, and similarly with $F(u^+).$  
This is a system of reaction-diffusion equations involving two-dimensional diffusion on the two sides of the membrane $\be$. As in \cites{dlajima,thin,zmarkusem2b}, in the limit the boundary  conditions  \eqref{eq.robin1} `jump into' the main equation to form new terms of the reaction part. If $c^-$ and $c^+$ are non-negative, the terms \[ - c^-u^- \mquad {and} - c^+u^+\] describe the process of random `killing' of diffusing particles on the upper and lower sides, respectively. On the parts where $c^-$ and $c^+$ are negative, these terms model inflow of new diffusing particles. 
However, it is the matrix 
\begin{equation}\label{macierz} \begin{pmatrix} - \alpha & \alpha \\ \beta & - \beta \end{pmatrix} \end{equation}
featuring in  \eqref{intro:2}, that constitutes the most most intriguing part of the limit system.
 As already explained, $\alpha $ and $\beta$ in \eqref{eq.robin2} should be interpreted as permeability coefficients of the membrane. Here, they play a similar role: they should be interpreted as \emph{jump intensities}: particles diffusing on, say, the lower side of the membrane, may jump to the upper side; in regions where $\alpha $ is large, expected times to such jumps are shorter than those in the regions where $\alpha$ is small. Similarly, a particle diffusing on the region of the upper side where $\beta$ is large will jump to the lower side on average quicker than a similar particle diffusing in a region where $\beta$ is small. Remarkably, if $\alpha=\beta=0$, \eqref{eq.robin2} reduces to Neumann boundary conditions: the membrane is non-permeable, and  diffusing particles in lower and upper layers never filter to the other side. A similar observation may be made concerning \eqref{intro:2}: for $\alpha=\beta=0$, the system is uncoupled: the are no jumps between the lower and upper sides.

\subsection{A view through a magnifying glass} \label{avtamg}

 To see that solutions to \eqref{eq.diff1}--\eqref{eq.robin2} gradually lose dependence on $z$ variable, we look at $\Omega_\eps$ through a magnifying glass, by  introducing the  change of variables, $\tilde{z} = \eps^{-1}z$, which transforms $\Omega_\eps$ into \[\Omega:=\Omega_1. \]
To this end, we write $\tilde{u}(t,x,y,\tilde{z}) = u(t,x,y,\eps^{-1}z)$. Then equation \eqref{eq.diff1} transforms to 
\begin{align}
\partial_t \tilde{u}(t,x,y,\tilde{z}) &= (\partial_x^2  +  \partial_y^2 + \eps^{-2}\partial_z^2) \tilde{u}(t,x,y,\tilde{z})  + F(\tilde{u} (t,x,y,\tilde{z}) ), \label{eq.diff2} \end{align}
for $(x,y,\tilde{z}) \in \Omega, t >0 $ while the boundary conditions \eqref{eq.robin1} become
\begin{align}
\partial_{\tilde{z}} \tilde{u}(t,x,y,-1) &= \eps^2 c^-(x,y)\tilde{u}(x,y,-1) , \label{eq.boundary}\\
\partial_{\tilde{z}} \tilde{u}(t,x,y,1) &= - \eps^2 c^+(x,y)\tilde{u}(x,y,1), \qquad x,y\in \R, t >0. \nonumber  \end{align}
Similarly, the transmission conditions become 
\begin{align}
\partial_{\tilde z} \tilde u(t,x,y,0-)&= \eps^2 \alpha (x,y)[\tilde u(x,y,0-)-\tilde u(x,y,0+)],  \label{eq.boundary2} \\  
\partial_{\tilde z} \tilde u(t,x,y,0+)&= \eps^2 \beta (x,y)[\tilde u(x,y,0+)-\tilde u(x,y,0-)],
 \qquad x,y \in \R, t >0. \nonumber
\end{align}

\subsection{A generation theorem in $L^2(\Omega)$}\label{agtinl} 

For notational simplicity we drop the tildes, and then rewrite system \eqref{eq.diff2}-\eqref{eq.boundary2} as an abstract evolution equation on the space $L^2(\Omega)$, as follows:
\begin{equation}\label{eq.cauchyproblemf}
\partial_t u_\eps(t) = A_\eps  u_\eps (t) + F(u_\eps (t)), \qquad u_\eps (0) = \overset{\text{o}}u\in L^2(\Omega)
\end{equation}
where $u_\eps : [0,\infty) \to L^2(\Omega)$ and $A_\eps$ is a suitable realization of the differential operator
$\partial_x^2+\partial_y^2+\eps^{-2}\partial_z^2$, subject to the boundary and transmission conditions \eqref{eq.boundary}-\eqref{eq.boundary2} (see later on). The reaction term $F$, although denoted by the same letter as the function featuring in \eqref{eq.diff1}, has a slightly different meaning. Namely, for a $u\in L^2(\Omega)$ we may define 
\begin{equation}\label{ef} 
\left ( \mathsf F(u) \right ) (x,y,z) := F(u(x,y,z))\end{equation}
where $F$ on the right-hand side is the function from \eqref{eq.diff1}. Assuming that $F(0)=0$ or, more generally, that there is a $u\in L^2(\Omega)$ such that $\mathsf F(u) \in L^2(\Omega)$, we check, using the existence of a global Lipschitz constant for $F$, that \eqref{ef} defines a globally Lipschitz continuous map $L^2(\Omega)\to L^2(\Omega)$, with the Lipschitz constant inherited from $F$. In \eqref{eq.cauchyproblemf}, for simplicity of notation,  we do not distinguish between $F$ and $\mathsf F$. 

As discussed in \cite{zmarkusem2b}, in dealing with well-posedness and convergence of solutions to \eqref{eq.cauchyproblemf}, it is a good strategy to work first with the related problem without the nonlinear term (see also the end of Section \ref{zbieznosc}):
\begin{equation}\label{eq.cauchyproblem}
\partial_t u_\eps(t) = A_\eps  u_\eps (t), \qquad u_\eps (0) = \overset{\text{o}}u\in L^2(\Omega),
\end{equation}
and we will follow this path. To this end, we start by establishing well-posedness of the problem \eqref{eq.cauchyproblem} making use of the theory of \emph{sesquilinear forms} (see e.g. \cites{ouhabaz,kato} for details of this theory).

 We recall that if $H$ is a complex Hilbert space, a \emph{sesquilinear form}
on $H$ is a mapping $\mathfrak{a}: D(\mathfrak{a}) \times D(\mathfrak{a}) \to \C$ which is linear in the first component and antilinear in the second component. 
A form $\mathfrak{a}$ is called \emph{accretive} if
$\Re \mathfrak{a}[u,u] \geq 0$ for all $u\in D(\mathfrak{a})$; it is called \emph{closed}, if $D(\mathfrak{a})$ is a Hilbert space with respect to the inner product $[u,v]_\mathfrak{a} = \Re \mathfrak{a}[u,v] + [u,v]_H$. A sesquilinear form is called \emph{densely defined}, if $D(\mathfrak{a})$ is dense in $H$. It is called \emph{symmetric}, if $\mathfrak{a}[u,v] = \mathfrak{a}[v,u]$. As customary, we will write $\mathfrak a[u]$ for $\mathfrak a[u,u]$.  

Given an accretive and closed sesquilinear form $\mathfrak{a}$ that is densely defined, we can define the associated operator $A$ by setting 
\[
D(A) = \{ u\in D(\mathfrak{a}) : \exists\, f\in H : \mathfrak{a}[u,v] = -[f, v]_H\,\,\forall\, v\in D(\mathfrak{a})\} 
\]
and $Au\coloneqq f$. We thus have $\mathfrak{a}[u,v] = -[Au, v]_H$ for all $u\in D(A)$ and $v\in D(\mathfrak{a})$.
It is well known that the operator $A$ associated to an accretive, densely defined and closed sesquilinear form, is the generator of an analytic contraction semigroup on the space $H$.

Our goal in this section is  to find a sesquilinear form $\mathfrak{a}_\eps$ such that \begin{equation} \label{aeps} A_\eps  \coloneqq
\partial_x^2+\partial_y^2+\eps^{-2}\partial_z^2\end{equation} with boundary and transmission conditions \eqref{eq.boundary}--\eqref{eq.boundary2} is its associated operator. 
To this end, we define $\ha \subset L^2(\Omega)$ by
\[
\ha 
\coloneqq \{ v \in L^2(\Omega) : v_{|\Omega^+} \in H^1(\Omega^+),   v_{|\Omega^-} \in H^1(\Omega^-) \}, 
\]
where $\Omega^+\coloneqq  \Omega^+_1$ and $ \Omega^-\coloneqq  \Omega^-_1$; this is going to be the domain of the forms $\mathfrak{a}_\eps, \eps >0.$ We recall (see e.g., \cite{adams}*
{Part I, Case C of Theorem 4.12}) that 
each $w \in  H^1(\Omega^+)$ leaves square integrable traces on $\buen\coloneqq \buej $ and $\bpen$ (denoted $w(x,y,1)$ and $w(x,y,0+)$, respectively), and that the trace operators are continuous. Similarly, each $w \in  H^1(\Omega^-)$ leaves square integrable traces on $\blen \coloneqq \blej $ and $\bmen $ (denoted $w(x,y,-1)$ and $w(x,y,0-)$, respectively), and the trace operators are bounded. Hence, a $v \in \ha$ leaves square integrable traces on each of these four sets, and we have the following four bounded trace  operators: 
\begin{equation} \begin{matrix} T^{\text{up}}: \ha \to L^2 (\buen),  & T^{\text{lo}}: \ha \to L^2 (\blen),\\  T^{+}: \ha \to L^2 (\bpen), & T^{-}: \ha \to L^2 (\bmen),\end{matrix}   \label{slady} \end{equation}
where $\ha$ is equipped with the norm 
\[ \|v\|_\ha = \sqrt { \|v_{|{\Omega^+}}\|_{H^1(\Omega^+)}^2 +  \|v_{|{\Omega^-}}\|_{H^1(\Omega^-)}^2}.\]

Finally, let \[ D_\eps \subset L^2(\Omega)\] be composed of $u$ such that 
$v_{|\Omega^+} \in H^2(\Omega^+),$    $v_{|\Omega^-} \in H^2(\Omega^-) $ and such that boundary and transmission conditions \eqref{eq.boundary}--\eqref{eq.boundary2} are satisfied in the weak sense (see e.g. \cite{zmarkusem}*{Section 4.2} for details). For such $u$, we define $A_\eps  u $ as $(\partial_x^2+\partial_y^2+\eps^{-2}\partial_z^2)u$.  

With these definitions, we take $u\in D_\eps$ and $v\in \ha$, multiply $A_\eps  u$ by $\bar v$, and integrate the product over $\Omega^+$. 
Integration by parts formula then shows that  
\begin{align}
\int_{\Omega^+} (A_\eps  u) \bar{v}\, & (x,y,z)\di (x,y,z) \notag \\ &=  
-\int_{\Omega^+} \left [ \partial_x u\partial_x\bar v + \partial_y u \partial_y\bar v + \eps^{-2}\partial_zu \partial_z\bar v\, \right ] (x,y,z) \di (x,y,z) \notag\\
&
\quad + \eps^{-2}\int_{\be} \partial_zu(x,y,1) \bar v(x,y,1)  \di (x,y)\notag\\
&\quad - \eps^{-2}\int_{\be}\partial_z u(x,y,0+)\bar v(x,y,0+)\,   \di (x,y)  \notag 
\\& =  - \int_{\Omega^+} \left [ \partial_x u\partial_x\bar v + \partial_y u \partial_y\bar v + \eps^{-2}\partial_zu \partial_z\bar v\right ] \,\di (x,y,z)\notag\\
&\quad 
- \int_{\be} c^+(x,y)u(x,y,1)\bar v(x,y,1) \,\di (x,y)  \label{eq.a}
\\
&\quad 
- \int_{\be} \beta (x,y)[u(x,y,0+)- u(x,y,0-)] \bar v(x,y,0+) \,\di (x,y).  \notag
\end{align}
Here, in the second step, we used the first of the boundary conditions \eqref{eq.boundary} and the first of the transmission conditions \eqref{eq.boundary2}. Similarly, we check that 
\begin{align}
\int_{\Omega^-} (A_\eps  u) \bar{v} &(x,y,z) \,  \di (x,y,z)\nonumber \\
& =  - \int_{\Omega^-} \left [ \partial_x u\partial_x\bar v + \partial_y u \partial_y\bar v + \eps^{-2}\partial_zu \partial_z\bar v\right ] (x,y,z)\,\di (x,y,z)\notag\\
&\quad 
- \int_{\be} c^-(x,y)u(x,y,-1)\bar v(x,y,-1) \,\di (x,y) \notag \\
&\quad 
- \int_{\be} \alpha (x,y)[u(x,y,0-)- u(x,y,0+)] \bar v(x,y,0-) \,\di (x,y). 
\label{eq.b}
\end{align}
This calculation suggests that, for $u,v \in \ha$, we should define 
\begin{align} \mathfrak{a}_\eps [u,v] \notag &:= \int_{\Omega} \left [ \partial_x u\partial_x\bar v + \partial_y u \partial_y\bar v + \eps^{-2}\partial_zu \partial_z\bar v\right ] (x,y,z)\,\di (x,y,z)\\
&\phantom{:=}+ \int_{\be} c^+(x,y)u(x,y,1)\bar v(x,y,1) \,\di (x,y)  \notag
\\
&\phantom{:=}+ \int_{\be} c^-(x,y)u(x,y,-1)\bar v(x,y,-1) \,\di (x,y) \notag \\
&\phantom{:=} 
+ \int_{\be} \beta (x,y)[u(x,y,0+)- u(x,y,0-)] \bar v(x,y,0+) \,\di (x,y)  \notag\\
&\phantom{:=}+\int_{\be} \alpha (x,y)[u(x,y,0-)- u(x,y,0+)] \bar v(x,y,0-) \,\di (x,y). \label{forma} 
\end{align}

\begin{prop}\label{l.form}
Forms $\mathfrak a_\eps$ are densely defined and closed. Furthermore, there is a $\gamma>0$ such that  for all $\eps \in (0,1]$,
\begin{equation}\label{nierownosc}
|\Im (\mathfrak{a}_\eps +\gamma)[u] | \le \Re (\mathfrak{a}_\eps +\gamma)[u], \qquad u \in \ha.   \end{equation}
\end{prop}

\begin{proof}It is clear that $\ha$ is dense in $L^2(\Omega)$. 
Let (only in this proof)
\[ \mathfrak{b}_\eps [u,v]= \int_{\Omega} \left [ \partial_x u\partial_x\bar v + \partial_y u \partial_y\bar v + \eps^{-2}\partial_zu \partial_z\bar v\right ] (x,y,z) \,\di (x,y,z), \qquad u,v \in \ha \]
and 
\( \mathfrak{c}  = \mathfrak{a}_\eps - \mathfrak{b}_\eps \)
(note that $\mathfrak{c}$ does not depend on $\eps$). Then, $\mathfrak b_\eps$ is symmetric and, since $\eps \in (0,1]$, 
\[  \|\nabla u\|_{L^2(\Omega)}^2 = \mathfrak b_1 [u] \le  \mathfrak b_\eps [u ] \le \eps^{-2} \|\nabla u \|_{L^2(\Omega)}^2 .\]
It follows that the forms $\mathfrak b_\eps$ are accretive. They are also closed, since using this inequality  it may be shown that for each $\eps$, the norm induced by $\mathfrak b_\eps $ is equivalent to the norm in $\ha$.

Turning to analysis of $\mathfrak{c}$ we note first of all that it is bounded: there is a constant $C$ such that 
\[ |\mathfrak c [u,v]| \le C \|u\|_\ha \|v\|_\ha ;\] 
this is because 
all the trace operators \eqref{slady} are bounded and $c^-, c^+, \alpha$ and $\beta$ are essentially bounded functions. Moreover, since the boundary of $B$ is assumed to be Lipschitz continuous, all the trace operators \eqref{slady} are compact (see \cite{nevas}*{Thm 6.2, p. 103}). Hence, if a sequence  $(u_n)_{n\ge 1} $ of elements of $\ha$ converges to $0$ weakly, sequences of its traces converge strongly to zero in the corresponding $L^2$ spaces. Since $c^+, c^-, \alpha$ and $\beta$ are essentially bounded, it follows that $\gra \mathfrak c[u_n]  =0$. Hence, by Lemma 7.3 in \cite{dan13}, for each $\delta >0$ there exists a $c(\delta) >0$ such that 
\begin{equation} \label{bound} |\mathfrak c[u]| \le \delta \|u\|_\ha^2 + c(\delta) \|u\|_{L^2(\Omega)}^2  .\end{equation}
By Theorem VI.3.11 in \cite{kato}, this inequality combined with the fact that $\mathfrak b_\eps$ is closed, shows that so is $\mathfrak a_\eps = \mathfrak b_\eps + \mathfrak c .$ Moreover, taking $\delta = \frac 12 $ in \eqref{bound} we obtain,
for $\gamma = 2 c(\frac 12) + 1$, 
\begin{align*}
\max \{ |\Re \mathfrak c[u]|, |\Im \mathfrak c[u]|\} \le \frac 12 \mathfrak b_1 [u] + \frac \gamma 2 \|u\|^2_{L^2(\Omega)}\le \frac 12 \mathfrak b_\eps [u] + \frac \gamma 2 \|u\|^2_{L^2(\Omega)}.
\end{align*}
Thus 
\begin{align*}
|\Im \mathfrak a_\eps [u] | = |\Im \mathfrak c [u] |\le \frac 12 \mathfrak b_\eps [u] + \frac \gamma 2 \|u\|^2_{L^2(\Omega)}
\end{align*}
and
\begin{align*}
\Re \mathfrak a_\eps [u] \ge \mathfrak b_\eps [u] - |\Re \mathfrak c[u]| \ge \frac 12 \mathfrak b_\eps [u] - \frac \gamma 2\|u\|^2_{L^2(\Omega)}. 
\end{align*}
It follows that 
\[ |\Im \mathfrak a_\eps [u] | \le \Re \mathfrak a_\eps [u] + \gamma \|u\|^2_{L^2(\Omega)}.\]
Since $\Im (\mathfrak a_\eps + \gamma )[u] = \Im \mathfrak a_\eps [u]$ and $\gamma [u] = \gamma \|u\|^2_{L^2(\Omega)},$
this inequality is equivalent to \eqref{nierownosc}. 
\end{proof}

Inequality \eqref{nierownosc} shows in particular that forms $\mathfrak a_\eps + \gamma$ are accretive. Thus, the related operators are generators of holomorphic contraction semigroups. These operators are equal to $A_\eps  - \gamma I$ where $I$ is the identity operator in $L^2(\Omega)$, and $(A_\eps, D(A_\eps))$ is the operator related to the form $\mathfrak a_\eps$. Calculations \eqref{eq.a} and \eqref{eq.b} show  that $(A_\eps,D(A_\eps))$ is an extension of $(A_\eps, D_\eps).$ We may thus write 
\[ \|\e^{tA_\eps} \|  \le \e^{\gamma t}; \]
here, and in what follows, $A_\eps  $ is always considered with domain $D(A_\eps)$. However, inequality \eqref{nierownosc}, reveals much more: the forms $\mathfrak a_\eps +\gamma $ are uniformly holomorphic and so are the semigroups generated by $A_\eps- \gamma I$. This information will be of crucial importance in the next section.

\subsection{Convergence as $\eps \to 0$}\label{zbieznosc}

Finally, we want to let $\eps\to 0$. To that end, we use a convergence theorem for uniformly holomorphic forms due to E. M. Ouhabaz \cite{ouh95} which generalizes the convergence theorem of 
B. Simon \cite{simon}, who dealt with the case of symmetric forms. To explain: the forms \[\widetilde{\mathfrak a}_\eps:=\mathfrak a_\eps +\gamma ,\] in addition to being accretive and uniformly holomorphic, have the following properties: 
\begin{itemize}
\item [(a) ] they have the same domain (i.e., $\ha$) and $\Re \widetilde{\mathfrak a}_\eps [u]\le \Re \widetilde{\mathfrak a}_{\eps'} [u]$ for all $u\in \ha$, provided $\eps\ge \eps'$ (which is to say that $\Re \widetilde{\mathfrak a}_\eps [u]$ increases as $\eps $ decreases), 
\item [ (b) ] the imaginary parts of $\widetilde{\mathfrak a}_\eps [u]$ do not depend on $\eps$. 
\end{itemize}
Ouhabaz shows that in such a case (and in a more general situation), the form 
$\widetilde{\mathfrak b} [u] \coloneqq \lim_{\eps \to 0} \widetilde{\mathfrak a}_\eps [u]$ (extended via polarization formula), defined  on the domain 
\[ D(\widetilde{\mathfrak b}) = \{ u \in \ha; \sup_{\eps \in (0,1]} \widetilde{\mathfrak a}_\eps[u] < \infty \}\]
is accretive, closed and sectorial (so that \eqref{nierownosc} holds for all $u \in D(\widetilde{\mathfrak b})$, if $\mathfrak a_\eps$ is replaced by  $\mathfrak b:= \widetilde {\mathfrak b} - \gamma$). As we shall see soon, in our case, in distinction to the situation considered by Ouhabaz, this form is not densely defined. Hence, the related operator, say $B+\gamma I$ (where $B$ is the operator related to $\mathfrak b$), generates a holomorphic semigroup merely on the subspace $H_0:= \overline {D(\widetilde{\mathfrak a})}$. Nevertheless, Ouhabaz's arguments may be extended to this case to show that 
\[ \lim_{\eps \to 0} (\mu I - A_\eps  - \gamma I)^{-1} = (\mu I - B- \gamma I)^{-1} P, \]
strongly,  for all $\mu $ in a sector of the complex plane (see the comment on p. 676 in \cite{zmarkusem}).
Here, $P$ is the orthogonal projection onto $H_0$. Using straightforward arguments involving contour integrals, presented in more detail in e.g. \cite{deg} or \cite{knigaz}*{Chapter 31}, one then deduces that 
\( \lim_{\eps \to 0} \e^{-\gamma t} \e^{tA_\eps} = \e^{-\gamma t} \e^{tB} P \) or, simply, 
\begin{equation}\label{polgrupy} \lim_{\eps \to 0}  \e^{tA_\eps} = \e^{tB} P, \qquad t >0 \end{equation}
 (strongly). For $u\in H_0$ this limit is uniform for $t$ in compact subsets of $[0,\infty)$; for other $u$ it is uniform for $t$ in compact subsets of $(0,\infty)$.

Hence, we are left with the task of  characterizing $D(\widetilde{\mathfrak b})= D(\mathfrak b)$, the form $\mathfrak b$, and the related operator $B$. We want to check to see that the limit dynamics is governed by \eqref{intro:2}.

The only term in the definition of $\widetilde{\mathfrak a}_\eps [u]$ (see  \eqref{forma}) that involves $\eps$ is 
\[
 \eps^{-2}\int_\Omega |\partial_z u|^2 (x,y,z) \, \di (x,y,z);
\]
it is thus clear that $\sup_{\eps \in (0,1]} \widetilde{\mathfrak a}_\eps [u] <\infty$  implies that $\partial_z u = 0$ almost everywhere, i.e. that $u$ does not depend on $z$. Conversely, if $u$ does not depend on $z$, then the supremum in question exists, because no term in the definition depends on $\eps $. 
Thus, more specifically, for $u \in D(\widetilde{\mathfrak b})= D(\mathfrak b)$ there are functions $u^-, u^+ \in H^1 (\be)$ such that 
\begin{align*} 
u (x,y,z) &= u^-(x,y), \qquad z \in (-1,0),\\
u (x,y,z) &= u^+(x,y), \qquad z \in (0,1), 
\end{align*}
almost surely in $(x,y,z)$. Hence, any $u$ may be identified with a pair of elements of $H^1(\be)$ and 
$\overline {D(\mathfrak b)}$ may be identified with the direct sum of two copies of $L^2(\be)$.
Moreover, by polarization formula, for $u$ and $v$ in $D(\mathfrak b)$,  
\begin{align*}
\mathfrak b[u,v] & := \int_{\be} \left [ \partial_x u^+\partial_x\bar v^+ + \partial_y u^+ \partial_y\bar v^+ \right ] \,\di \lam_2 \\ &\phantom{=}+ \int_{\be} \left [ \partial_x u^-\partial_x\bar v^- + \partial_y u^- \partial_y\bar v^- \right ] \,\di \lam_2\\
&\phantom{:=}+ \int_{\be} c^+u^+\bar v^+ \,\di \lam_2  + \int_{\be} c^-u^-\bar v^- \,\di \lam_2 \notag \\
&\phantom{:=} 
+ \int_{\be} \alpha (u^- - u^+) \bar v^- \,\di \lam_2  +\int_{\be} \beta (u^+- u^-) \bar v^+ \,\di \lam_2, \label{formal} 
\end{align*}
where $\lam_2$ is the two-dimensional Lebesgue measure. 
It is clear that the two terms in the third line here can be extended to the form defined on the entire $L^2(\be) \times L^2(\be)$, and that the associated operator is bounded and given by:
\[ \binom {u^-}{u^+} \mapsto - \binom {c^-u^-}{c^+u^+}. \]
We see the same operator also in \eqref{intro:2}. Similarly, the terms in the fourth line come from the bounded operator in $L^2(\be) \times L^2(\be)$ which may be represented by the matrix \eqref{macierz}. Moreover, the first term is well-known: the associated operator is the Neumann Laplace operator  $\Delta_{2D}$ (\ie the $2D$ Laplace operator with Neumann boundary conditions) in $L^2(\be)$ -- see e.g. Example 8.1.6 in \cite{arendtnotes}, and an analogous remark concerns the second term. Thus, the first two terms are associated with the operator 
\[ \binom {u^-}{u^+} \mapsto \binom {\Delta_{2D} u^-}{\Delta_{2D} u^+}, \] 
and the operator related to the entire limit form may be represented as 
\[ B= \begin{pmatrix} \Delta_{2D} -c^-& 0 \\ 0 & \Delta_{2D} -c^+ \end{pmatrix} + \begin{pmatrix} - \alpha & \alpha \\ \beta & - \beta \end{pmatrix}. \]
In other words, formula \eqref{polgrupy} shows that mild solutions of \eqref{eq.diff1}-\eqref{eq.robin2} with initial condition $\overset{\text{o}}u\in L^2(\Omega)$, and non-linear term equal $0$, 
converge to those of \eqref{intro:2} with initial condition $P\overset{\text{o}}u = (P^- \overset{\text{o}}u, P^+\overset{\text{o}}u)$ where 
\begin{equation}\label{rzuty} P^- \overset{\text{o}}u (x,y) = \int_{-1}^0 \overset{\text{o}}u(x,y,z) \ud z \qquad P^+ \overset{\text{o}}u (x,y) = \int_0^1 \overset{\text{o}}u(x,y,z) \ud z. \end{equation}

Hence, it remains to take care of the non-linearity. By the main theorem of \cite{zmarkusem2b}, however, convergence of semigroups $\sem{A_\eps}$, even in a degenerate manner, as in \eqref{polgrupy}, implies convergence of mild solutions of  
 \eqref{eq.diff1}-\eqref{eq.robin2} to solutions of
\[ \partial_t u (t) = Bu (t) + PF(u(t)).\]
Moreover, since $F$ (or, in fact, $\mathsf F$, see \eqref{ef}) leaves $H_0$ invariant, $P$ is not needed on the right-hand side here: we have proved that mild solutions of \eqref{eq.diff1}-\eqref{eq.robin2} with initial condition $\overset{\text{o}}u\in L^2(\Omega)$ 
converge to those of \eqref{intro:2} with initial condition $P\overset{\text{o}}u$.

\section{Analysis in $C$}\label{aic}

\subsection{The main result; robustness}\label{tmrr} 

For an analogous result in the space of continuous functions we need more restrictive assumptions on the base $\be$: we assume that $\be$ is a bounded, connected and open set, and that its boundary $\partial \be$ is of class $C^{2,\kappa}, \kappa \in (0,1]$ (see \cite{ethier}*{p. 368}). This allows concluding that the Neumann Laplace operator with suitable domain in $C(\be)$ is closable and its closure generates a Feller semigroup in $C(\be)$ (\cite{ethier}*{p. 369}). For technical reasons we also need to be content with constant permeability coefficients $\alpha$ and $\beta$. On the other hand,  we generalize the mechanism of filtering through the membrane. More specifically, given parameters $p,q\in [0,1]$ and Borel probability measures $\mu$ and $\nu$ on $[-1,0]$ and $[0,1]$, respectively, we consider the following transmission conditions describing permeability of  the membrane separating the lower and upper layers:
{\small
\begin{align}
(\eps p\partial^2_z  + (1-p)\partial_z) u(t,x,y,0-)&= \eps \alpha \left [ \int_{[0+,\eps]} u(t,x,y,z)\nu_\eps(\udd z)- u(t,x,y,0-)\right ], \nonumber \\ 
(\eps q\partial^2_z  - (1-q)\partial_z) u(t,x,y,0+)&= \eps \beta \left [\int_{[-\eps,0-]}u(t,x,y,z)\mu_\eps(\udd z)-u(t,x,y,0+)\right ], 
 \label{C:eq.robin2} 
\end{align}}
\hspace{-0.3cm} where $(x,y) \in \be, t >0$, $\mu_\eps $ is the transport of $\mu$ to $[-\eps,0-]$ (via the map $x\mapsto -\eps x$), and $\nu_\eps $ is the transport of $\nu $ to $[0+,\eps]$ via $x\mapsto \eps x.$  Again, epsilons in these transmission conditions are arranged in such a way that the limit as $\eps \to 0$ is non-trivial. Remarkably, if $p=q=1$ the epsilons cancel out, \ie no scaling is needed.

In comparison to \eqref{eq.robin2}, these boundary conditions describe a more general way Brownian particles on one side of the membrane may filter to the other side (see also Section \ref{tvc:agt}): the additional term with the second derivative speaks of the possibility for a particle to stick to the membrane for some random time. In particular, for $p=1$ the particles are stopped at the lower part of the membrane, and after an exponential time (with parameter $\alpha$) released to jump to the upper side. The measures $\mu_\eps $ and $\nu_\eps$  describe a random position of a particle after it filters from one side of the membrane to the other. For $p=q=0$ and $\mu=\delta_{0},\nu=\delta_{0}$ where $\delta_{0}$ is the Dirac measure concentrated at $0$, transmission conditions \eqref{C:eq.robin2}  reduce to \eqref{eq.robin2}.  

As in Section \ref{avtamg}, instead of working with `the same equation' but in varying  spaces $C(\Omega_\eps), \eps \in (0,1]$ of continuous functions on thiner and thiner domains $\Omega_\eps = \Omega_\eps^- \cup \Omega_\eps^+$ where, this time,  
\begin{align*}
 \Omega_\eps^- &\coloneqq \{ (x,y,z)\in \R^3 : (x,y) \in \be, -\eps \le z \le 0-\},\\
 \Omega_\eps^+ &\coloneqq \{ (x,y,z)\in \R^3 : (x,y) \in \be, 0+ \le z \le \eps\}, 
\end{align*}
we blow up the thin coordinate $z$ by dividing it by $\eps$, to work with a family of equations in a single reference space \[ C(\Omega)\coloneqq C(\Omega_1).\] Notably, this transformation is an isometric isomorphism of the spaces of continuous functions involved (but not of the $L^2$-type spaces considered in Section \ref{ldwa}): we are using a magnifying glass and not a distorting mirror.

\newcommand{\naeps}{\mathfrak A_\eps}
\newcommand{\onaeps}{\overline{\mathfrak A_\eps}}
The family of equations in $C(\Omega)$ we obtain is of the form
\begin{equation}\label{cauchy} \partial_t u(t) = \onaeps u(t), \qquad u(0)=\overset{\text{o}}u\in C(\Omega),  \end{equation} 
where $\naeps  = \partial_x^2+\partial_y^2+\eps^{-2}\partial_z^2$ and $\dom{\naeps}$ is composed of $u\in C(\Omega)$ such that (a) when restricted to either of $\Omega^-$ or $\Omega^+$ they are of class $C^2$ and for each $z\in [-1,0-]\cup[0+,1]$, $(x,y)\mapsto u(x,y,z) $ is of class $C^{2,\kappa}$ and (b) besides Neumann boundary conditions on the boundary of $\Omega$, they satisfy the following transmission conditions on the membrane separating the upper and lower parts of $\Omega$: 
\begin{align}
(p\partial^2_z  + (1-p)\partial_z) u(x,y,0-)&= \eps^2 \alpha \left [ \nu_{x,y} (u)- u(x,y,0-)\right ], \nonumber \\ 
(q\partial^2_z  - (1-q)\partial_z) u(x,y,0+)&= \eps^2 \beta \left [\mu_{x,y} (u)-u(x,y,0+)\right ], 
 \label{C:eq.robin3} 
\end{align}
where $(x,y)\in \be $, whereas $\nu_{x,y} (u) $ and $\nu_{x,y} (u)$ are shorthands for \[ \int_{[0+,1]} u(x,y,z)\nu (\udd z) \mquad{ and } \int_{[-1,0-]}u(x,y,z)\mu(\udd z),\]
respectively. As we shall see later,  $\naeps, \eps \in (0,1]$ are closable and  that their closures $\onaeps, \eps \in (0,1]$ are conservative Feller generators (see Proposition \ref{tsgb:prop:1}). Our goal is to study the limit 
\( \grae \e^{t\onaeps}. \)


We note the absence of the non-linear term in \eqref{cauchy}. By the main  theorem of \cite{zmarkusem2b}  (already alluded to at the end of  Section \ref{ldwa}), a convergence result for semigroups implies also convergence of solutions of the related semi-linear equations (with globally Lipschitz continuous non-linearity), and thus we disposed of this term without loss of generality. Furthermore, we note that,  for simplicity of exposition, since the role of $c^-$ and $c^+$ of Section \ref{ldwa} has been already explained in \cites{dlajima,thin,zmarkusem2b}, and we want to focus on the more intriguing role of $\alpha$ and $\beta$, it is assumed that $c^-$ and $c^+$  are now zero (since Neumann boundary conditions on the entire boundary of $\Omega$ are assumed).

As a preparation for the main theorem in this section, let $\notzee \subset C(\Omega)$ be the subspace of $u\in C(\Omega)$ that do not depend on $z$. For each member $u$ of $\notzee$ there are two continuous functions, say $u^-$ and $u^+$, on $\be$ such that 
\begin{align} u(x,y,z) &= u^-(x,y), \qquad (x,y)\in \be, z \in [-1,0-],\nonumber \\
 u(x,y,z) &= u^+(x,y), \qquad (x,y)\in \be, z \in [0+,1],\label{tmrr:1} \end{align}
and thus $u$ may be identified with $(u^-,u^+)\in C(\be) \times C(\be).$ In other words, $\notzee$ is isometrically isomorphic with the latter Cartesian product.   By the generation theorem from p. 369 in \cite{ethier}, the Neumann Laplace operator, say  $\Delta_{2D}$, is closable and its closure  $\overline {\Delta_{2D}}$ generates a Feller semigroup in $C(\be)$. Thus, by the Phillips perturbation theorem, the operator 
\[ \mathfrak B= \begin{pmatrix} \overline{\Delta_{2D}} & 0 \\ 0 & \overline{\Delta_{2D}}  \end{pmatrix} + \begin{pmatrix} - \alpha & \alpha \\ \beta & - \beta \end{pmatrix}, \] 
with domain $\dom{\overline{\Delta_{2D}}} \times \dom{\overline{\Delta_{2D}}}$ is a generator in $C(\be) \times C(\be)$. (Also, if the latter space is appropriately identified 
with the space of continuous functions on two copies of $\be$, the operator $\mathfrak B$ may be seen to be a conservative Feller generator.) Our main theorem says that this $\mathfrak B$ governs the limit evolution. 
 
\begin{thm} \label{mainth} We have, 
\begin{equation}\label{mainth:1} \grae \e^{t\onaeps } u = \e^{t\mathfrak B} \mc P_{p,q}u, \qquad u \in C(\Omega),t >0\end{equation}
where $\mc P_{p,q}$ is the projection on $\notzee$ is defined by  
\[ \mc P_{p,q} u = (u_1,u_2) \]
with \begin{align} u_1(x,y) &= pu(x,y,0-)+ (1-p) \int_{-1}^0 u(x,y,z)\ud z, \notag \\ u_2(x,y) &= qu(x,y,0+)+ (1-q) \int_{0}^1 u(x,y,z)\ud z , \qquad (x,y) \in \be . \label{rzutyprim} \end{align}
\end{thm}

The intuition behind this theorem is as follows. As $\eps \to 0$, diffusion in the vertical direction becomes faster and faster, and the solutions to \eqref{cauchy} become more and more flat in this direction. Therefore, in the limit they resemble functions of two variables, defined on the two sides of the separating membrane. On each of these sides, we have diffusion (with reflection on the boundary), and these sides communicate via jumps, as expressed in the second matrix  defining the operator $\mathfrak B$. 

Our choice of transmission conditions allows examining  the phrase `solutions became flat' more closely. The trick works in agreement with our intuition because fast diffusion averages out solutions in the vertical direction: but the formula for $\mc P_{p,q}$ given above makes it clear that there are several ways this averaging may take place. As long as we are facing an inadhesive membrane, functions are averaged by means of \eqref{rzuty}, but a sticky membrane leads to \eqref{rzutyprim}. The approximating scheme is robust in the sense that the limit process does not depend on the mechanism of filtering through the membrane.  Nevertheless, the averaging that leads to this process does, and so does the limit equation: depending on $p$ and $q$ different initial conditions are needed.  

On the other hand, neither the limit semigroup nor the projection depend on the measures $\mu$ and $\nu$. It is thus irrelevant whether after filtering through the membrane a Brownian particle restarts it chaotic movement close to the vicinity of the membrane, on its other side (though this is the most natural choice) or somewhere further away. This effect is not surprising in view of the averaging property of diffusion, discussed briefly above.

The rest of this section is devoted to a step-by-step proof of Theorem \ref{mainth}, intertwined with a similar proof of the generation result.

\subsection{A building block: a holomorphic, Feller semigroup in $[0,1]$ and its asymptotic behavior} 
\label{abb}
(Sections \ref{abb} -- \ref{tvc} are devoted to the vertical component of the main semigroup, and thus we should think of the related functions (arguments of the semigroup) as depending on the $z$ variable. However, in the following analysis it will be more convenient to use $x$ as a variable instead. We will come back to using the coordinates of the previous section in Section \ref{dowod}.)

Let $C[0,1]$ be the space of continuous functions on the unit interval $[0,1]$, and let $C^2[0,1]$ be its subspace composed of twice continuously differentiable functions. Moreover, let for any $r\in [0,1]$, the operator $G_r$ be given by  
\[ G_r f = f'' \]
on the domain composed of $f \in C^2[0,1]$ such that 
\begin{equation}\label{app:1} rf'' (0) - (1-r) f'(0)= 0 \mquad{and} f'(1) = 0. \end{equation}
As we shall see in this section, $G_r$ is a generator of a conservative Feller semigroup \sem{G_r} (\ie  of a strongly continuous semigroup of positive contractions such that $\e^{tG_r}1_{[0,1]}=1_{[0,1]},$ where $1_{[0,1]}(x) = 1, x \in [0,1]$)  in $C[0,1]$. 

The process related to $G_r$ is a Brownian motion on $[0,1]$ with reflecting barrier at $x=1$ and a sticky barrier at $x=0$ (see \cite{liggett} p. 127), trapping Brownian particles for `an infinitely short time' (if $r\not = 1$). The duration of the imprisonment of the particles at $x=0$ depends on $r$: for $r=1$ the particles are for ever trapped at $x=0$ and for $r=0$ they are reflected. For intermediate $r$ the measure of the set of  times when a particle starting at $x=0$ is at $x=0$ again is of positive Lebesgue measure, and this measure increases with $r$ (\cite{liggett} p. 128). 
In what follows, $r$ will be referred to as a stickiness coefficient.   

As it transpires, as $t\to \infty$, a statistical equilibrium is reached between the particles trapped at $x=0$ and those evenly distributed across $[0,1]$ by diffusion. This fact is expressed in the following formula: 
\begin{equation}\label{app:2} \grat \e^{tG_r} f = P_r f \end{equation}
where 
\begin{equation}\label{app:2'}  P_rf(x) =rf(0) +  (1-r)\int_0^1 f(y)\ud y, \qquad x \in [0,1].\end{equation} We prove \eqref{app:2} in Theorem \ref{ass}, further down. We start our analysis with the generation result. 

\begin{prop} \label{toiac} The operator $G_r, r \in [0,1]$ is a  conservative Feller generator. 
\end{prop} 
\begin{proof}The argument presented in \cite{knigaz} p. 17 shows that $G_r$ satisfies the positive maximum principle. It is also clear that $G_r$ is densely defined, that $1_{[0,1]}$ belongs to $\dom{G_r}$ and that $G_r 1_{[0,1]}= 0.$ Therefore, by the Hille--Yosida theorem for Feller semigroups (\cite{ethier}, Thm. 2.6 p. 13 and Thm. 2.2 p. 165, or \cite{kniga} Thm. 8.3.4, p. 328) it suffices to 
check the range condition: for any $g\in C[0,1]$ and $\lam >0$ there is an $f \in \dom{G_r}$ such that 
\begin{equation}\label{app:3} \lam f - G_r f = g .\end{equation}
(In particular, existence of solutions to the resolvent equation for one $\lam >0$ and all $g$ implies that $G_r$ is closed, see e.g. \cite{ethier} Lemma 2.2, p. 11.) 
To this end, we recall that $h = h_{\lam, g}$ defined by 
\[ h(x) = \frac 1{2\sqrt \lam } \int_0^1 \e^{-\sqrt \lam |x-y|} g(y) \ud y, \qquad x \in [0,1],  \]
belongs to $C^2[0,1]$ and satisfies $\lam h - h'' = g$. Therefore, for any constant $C$, the same is true about $f\in C^2[0,1]$ given by 
\begin{equation}\label{app:4} f(x) = C\cosh \sqrt \lam (1-x) - h(1) \sinh\sqrt \lam (1-x) + h(x) , \qquad x\in [0,1]. \end{equation}
Since $h'(1) = - \sqrt\lam h(1), $ we have $f'(1) = 0$. Moreover, the first condition in \eqref{app:1} is satisfied iff 
\begin{equation} \label{app:5} C= \frac {r[\lam h(1) \sinh \sqrt \lam + g(0) - \lam h(0)] + (1-r) [ h(1) \sqrt \lam \cosh \sqrt \lam + \sqrt \lam h(0)]}{\lam r \cosh \sqrt \lam + (1-r) \sqrt \lam \sinh \sqrt \lam }. \end{equation} 
Since $f$ with so-defined $C$ belongs to $\dom{G_r}$, we are done. \end{proof}

For the study of the asymptotic behavior of the semigroups generated by $G_r, r \in [0,1]$ we need a number of auxiliary results, presented below. 
The first of these reveals that each $\sem{G_r}$ is more regular than an ordinary Feller semigroup. 

\begin{prop}\label{cosinusy} Operators $G_r, r \in [0,1]$ are generators of strongly continuous cosine families in $C[0,1]$. 
\end{prop}

\begin{proof} This is a particular instance of a theorem due to Xiao and Liang \cites{xiao03,xiao08}. Since the proof of this general theorem is quite involved, and our case is rather simple,  for completeness, we sketch a straightforward proof based on the method of images (see \cites{kosinusy,kelvin}). This method leads to a semi-explicit formula for the cosine family, and we will use this formula later.  

Given a continuous function $f$ on $[0,1]$ we extend it to the interval $[0,2]$ by symmetry about $x=1$, by defining \begin{equation} \label{ogag:1} f(x) = f(2-x)\end{equation} 
for $ x \in [1,2].$
If $f$ is twice continuously differentiable on $[0,1]$ and $f'(1)=0$, the so-extended function is twice continuously differentiable on $[0,2]$.    
Next, if $r\in [0,1)$, we extend $f$ to $[-2,2]$ by agreeing that  (comp. \cite{kosinusy}, eq. (2.4))
\begin{equation}\label{ogag:2} f(-x) = 2\e^{-\kappa t}f(0)+ \kappa \int_0^x \e^{-\kappa (x-y)} f(y) \ud y - f(x),\end{equation}
for $ x\in [0,2]$,
where $\kappa \coloneqq \frac{1-r}r$,  and note that for $r=1$ this is an odd extension of $f$:   we have \[f(-x)=2f(0)-f(x)\] (comp. \cite{liggett} p. 125). If $r=0$, we take $f(-x)=f(x)$ (symmetry about $x=0$). These extensions are chosen so that $f$ is twice continuously differentiable provided $f\in \dom{G_r}$ (see \cite{kosinusy} pp. 667 and 674). Having defined (an extension of) $f$ on $[-2,2]$ we may extend its definition to $[-2,4]$ by formula \eqref{ogag:1}, and then again to $[-4,4]$ by formula \eqref{ogag:2}. Continuing this procedure of repeated reflections (see \cite{feller} p. 341 or \cite{bobgremur}) we construct a twice continuously differentiable function on the entire line such that \eqref{ogag:1} and \eqref{ogag:2} are true for all $x\ge 1$ and all $x\ge 0$, respectively.    
This allows defining the family $(C_r(t))_{t\in \R}$ of operators in $C[0,1]$ by 
\begin{equation}\label{cosinusy:1}
C_r(t)f(x) = \frac 12 (f(x+t) + f(x-t) ), \qquad x \in [0,1], t \in \R;\end{equation}
(note that the extension of $f$ depends on $r$ and so do these operators). 
The extension of $f$ is chosen in such a way that $C_r(t)f \in \dom{G_r}$ for all $t\in \R$ provided $f\in \dom{G_r}$. It follows that $(C_r(t))_{t\in \R}$ is a cosine family (see \cite{kosinusy} for details). Moreover, since for $f\in \dom{G_r}$, the extension of $f$ constructed above is twice continuously differentiable on $\R$, 
$\lim_{t\to 0} \frac {2(C_r(t)f - f)}{t^2} = f''$ for $f \in \dom{G_r}$.   Therefore, the generator of $(C_r(t))_{t\in \R}$ extends $G_r$. However, since (by Proposition  \ref{toiac}) the range of $\lam - G_r$ is the entire $C[0,1]$ no cosine family generator can be a proper extension of $G_r$, and we conclude that the generator of $(C_r(t))_{t\in \R}$ is $G_r$. 
\end{proof}

\begin{lem}\label{irr} Let $r \in [0,1)$. Then, the semigroup $\sem{G_r}$ is irreducible: for any $\lam >0$ and $g\ge 0$ the solution to the resolvent equation \eqref{app:3} is strictly positive. 
\end{lem}
\begin{proof} The idea lying behind the following proof is that the transition probabilities of the process related to $\sem{G_r}$ are larger than those of the minimal process in which a particle reaching $x=0$ is killed and removed from the state space.  Nevertheless, the argument is purely `analytic'. 

As a bit of algebra shows, 
\[ \frac {r[x h(1) \sinh x - x h(0)] + (1-r) [ h(1) \cosh x + h(0)]}{x r \cosh x + (1-r)  \sinh x}> \frac {h(1)\sinh x - h(0)}{\cosh x },\]
for all $x>0$ and $r \in [0,1)$ (for $r=1$ this turns into equality). Hence, even if $g(0)=0$, $f$ defined by \eqref{app:4} and \eqref{app:5} satisfies $f(0)>0$. By the same token, $C$ of \eqref{app:5} is larger than 
\[ C_0 \coloneqq \frac {h(1)\sinh \sqrt \lam  - h(0)}{\cosh \sqrt \lam },\]
and thus $f$ is larger than $f_0$, where $f_0$ is defined by \eqref{app:4} with $C$ replaced by $C_0$. 

To show that $f_0(x) >0 $ for all $x \in (0,1]$ we first note that 
\[ f_0(x) = \frac {\sinh \sqrt \lam  x}{\cosh \sqrt \lam } h(1) - \frac{\cosh \sqrt \lam (1- x)}{\cosh \sqrt \lam } h(0) + h(x). \] 
By the definition of $h$ it follows that
\[ f_0(x) = \frac 1{4\sqrt \lam \cosh \sqrt \lam} 
\int_0^1 k_\lam (x,y) g(y) \ud y ,\]
where 
\begin{align*} 
k_\lam (x,y)&=2\cosh\sqrt \lam \e^{-\sqrt \lam |x-y|} + 2\sinh \sqrt \lam  x \e^{\sqrt \lam (y-1)} \\\ &\phantom{=}- 2\cosh \sqrt \lam (1- x) \e^{-\sqrt \lam y}\\
&=\e^{-\sqrt \lam (|x-y|+1)} + \e^{-\sqrt \lam (|x-y|-1)} + \e^{\sqrt \lam (x+y-1)}\\ &\phantom{=} - \e^{\sqrt \lam (y-x-1)} - \e^{\sqrt \lam (1-x-y)} - \e^{\sqrt \lam (x-1-y)}. \end{align*}
For $y\le x$, this expression reduces to 
\[ \e^{\sqrt \lam (1-x+y)} + \e^{\sqrt \lam (x+y-1)} - \e^{\sqrt \lam (1-x-y)} - \e^{\sqrt \lam (x-1-y)} \ge 0 \]
with equality holding only if $y=0$. Analogously, for $0<x < y \le 1$, it reduces to
\[ \e^{\sqrt \lam (1-y+x)} + \e^{\sqrt \lam (x+y-1)} - \e^{\sqrt \lam (y-x-1)} - \e^{\sqrt \lam (1-x-y)}>0\]
(since $x>0$). This shows that for each $x\in (0,1]$ the function $(0,1]\ni y \mapsto k_\lam (x,y)$ is continuous and strictly positive. Therefore, $f_0(x)>0$ for $x \in (0,1]$, and the proof is complete.  
 \end{proof}
 
\begin{lem}\label{lem:3} The domain $\dom{G_r}$, when equipped with the graph norm $\|f\|_{G_r}= \|f\|+\|f''\|$ where $\|\cdot \|$ is the norm in $C[0,1]$, embeds compactly into $C[0,1]$. \end{lem} 
\begin{proof} We are to prove that the unit ball in $\dom{G_r}$, when considered as a subset of $C[0,1]$ is relatively compact. To this end we note that members $f$ of this ball satisfy 
\[ \|f\|+\|f''\|\le 1 \mquad { and } f(x) = f(0) + f'(0)x + \int_0^x \int_0^y f''(z) \ud z \ud y , x\in [0,1],\]  
where $f'(0)= -\int_0^1 f''(y) \ud y$ (by the second part of the boundary conditions \eqref{app:1}). It follows that $|f'(0)|\le 1$ and then that  $|f(x)-f(y)|\le 2|x - y|, x,y\in [0,1]$ and thus these functions are equicontinuous. Hence, we are done by the 
Arzel\'a–Ascoli theorem.
\end{proof}

\begin{thm}\label{ass}There are positive constants $K$ and $\omega$ (depending perhaps on $r$) such that 
\begin{equation}\label{ass:1} 
 \|\e^{tG_r} - P_r\| \le K\e^{-\omega t}, \qquad t \ge  0, \end{equation}
where $P_r$ is defined in \eqref{app:2'}.  \end{thm}
\begin{proof} \ \ 
(i) The case $r=0$ is well-known (see e.g. \cite{knigaz} pp. 177-180). 

(ii) The case $r\in (0,1).$ Since $\dom{G_r}$ embeds compactly into $C[0,1]$ (by Lemma \ref{lem:3}), the resolvent operators $\rez{G_r}, \lam >0$ are compact (\cite{engel} p. 117). Also, since $G_r$ generates a cosine family (by Proposition \ref{cosinusy}), the Weierstrass formula implies that $\sem{G_r}$ may be extended to a holomorphic semigroup (of angle $\pi/2$, in the space of complex functions on $[0,1]$) -- see e.g. \cite{abhn} pp. 219--220. It follows that \sem{G_r} is immediately norm continuous (\ie $\lim_{s\to t} \|\e^{sG_r} - \e^{tG_r} \| =0, t >0$) -- this may be seen e.g. by combining Lemma 4.2 p. 52 and Theorem 5.2 (point (d)) p. 62 in \cite{pazy}. This together with compactness of the resolvent operators implies that also $\e^{tG_r}, t >0$ are compact (see \cite{engel} p. 117 or \cite{pazy} p. 48).  Finally,  by Lemma \ref{lem:3}, \sem{G_r} is  irreducible.

Therefore, all assumptions of the theorem in Section 3.5.1 of \cite{arendtsur} are satisfied. It follows that (i) the spectral bound \[ s(G_r) = \sup \{ \Re \lam: \lam \in \sigma (G_r)\},\]
where $\sigma(G_r)$ is the spectrum of $G_r$, is larger than $-\infty$, and (ii) there are positive constants $K$ and $\omega$ and  a non-zero operator $P_r$ such that 
\begin{equation}\label{ass:2}  \|\e^{-s(G_r) t}\e^{tG_r} - P_r \|\le K \e^{-\omega t}, \qquad  t\ge  0.\end{equation}
Since \sem{G_r} is a contraction semigroup, $s(G_r)\le 0$, and since $\e^{tG_r} 1_{[0,1]}=1_{[0,1]},$ $s(G_r)$ cannot be strictly negative.  Hence, $s(G_r) =0$ and to prove \eqref{ass:1} we only need to show that $P_r$ in \eqref{ass:2} is given by \eqref{app:2'}. 

To this end, recall that existence of the limit $\grat \e^{tG_r}g, g \in C[0,1]$ implies existence of $\lim_{\lam \to 0} \lam \rez{G_r}g $ and the two then coincide. Moreover, a limit in the norm, when it exists, 
must of course coincide with the pointwise limit. 
On the other hand, $\lim_{\lam \to 0} \lam \rez{G_r}g(x)$ is easy to calculate, since we know the exact form of  $f(x) = \rez{G_r}g(x)$; it is given in \eqref{app:3} and \eqref{app:4}. Namely, it is easy to see that $\lim_{\lam \to 0} \lam \rez{G_r}g(x) = \lim_{\lam \to 0} \lam C$ for the $C=C_{\lam,g}$  of \eqref{app:4}. Moreover, when multiplied by $\lam $ this $C$ converges to $rg(0)+ (1-r)\int_0^1g(x) \ud x$, as $\lam \to 0$ (note that $h$ appearing in the definition of $C$ also depends on $\lam$). This completes the proof. 

(iii) In the case $r=1$, the semigroup $\sem{G_r}$ is not irreducible (see the proof of Lemma \ref{irr} -- the solution $f$ to the resolvent equation equals $0$ at $x=0$ as long as $g(0)=0$), and we need to proceed differently. Fortunately, the very fact that \sem{G_1} is not irreducible suggests a different line of attack. The cosine family $\left (C_1(t)\right )_{t\in \R}$ 
constructed in Proposition \ref{cosinusy} leaves the subspace $C_0(0,1]=\{f \in C[0,1]; f(0)=0\}$ of $C[0,1]$ invariant: if $f(0)=0$ then $C_1(t)f (0)=0$ for all $t\in \R$, because the graph of the extension of $f$ featuring in \eqref{cosinusy:1} is antisymmetric about $x=0$. The generator, say $G_1^0$, of  the restriction of $\left (C_1(t)\right )_{t\in \R}$ (and of the restriction of $\sem{G_1}$) to $C_0(0,1]$ is the part of $G_1$ in this subspace, \ie $G_1^0$ is the operator of the second derivative on the domain \(\dom{G_1^0}= \{f\in C_0(0,1]\cap C^2[0,1], f'' \in C_0(0,1]\}\) or, equivalently, $\dom{G_1^0}= \{f\in C_0(0,1]\cap C^2[0,1], f''(0)= 0\}.\)  Also, for any $g\in C_0(0,1]$ the function $f(x) = \int_0^x \int_0^y g(z) \ud z \ud y , x \in [0,1]$ belongs to $\dom{G_1^0}$ and we have $G_1^0 f=g.$ It follows that $0$ belongs to the resolvent set of $G_1^0$, and, since \sem{G_1^0} is a positive contraction semigroup, Proposition 3.11.2 in \cite{abhn} implies that $s(G_1^0)<0.$ Therefore, see \cite{arendtsur} p. 13, there are positive constants $K$ and $\omega $ such that 
\[ \|\e^{tG_1^0} \|_{\mathcal L (C_0(0,1])} \le K \e^{-\omega t}, \qquad t \ge 0.  \]
On the other hand, $\sem{G_1}$ being conservative, given $f \in C[0,1]$ we may consider $f_0 \coloneqq f - f(0)1_{[0,1]} \in C_0(0,1]$ and write 
\[ \|\e^{tG_1} f - f(0)1_{[0,1]}\|=\| \e^{tG_1} f_0 \| = \| \e^{tG_1^0} f_0 \| \le K \e^{-\omega t} \|f_0\| \le 2 K \e^{-\omega t} \|f\|.\]
This completes the proof. \end{proof}

Before completing this section, we note that $C[0,1]$ is isometrically isomorphic to $C[-1,0]$, the space of continuous functions on $[-1,0]$, with positivity preserving isometric isomorphism $I: C[-1,0] \to C[0,1]$ given by $If(x)=f(-x), x\in [0,1]$. The operator $G_r^I$ given by $G_r^I f = f''$ on the domain composed of twice continuously differentiable functions $f \in C[-1,0]$ such that 
\begin{equation}\label{app:1'} rf'' (0) + (1-r) f'(0)= 0 \mquad{and} f'(-1) = 0 \end{equation} 
is the image of $G_r$ in $C[-1,0]$. This is to say that $f $ belongs to $\dom{G_r^I}$ iff $If $ belongs to $\dom{G_r}$ and we have $G_r I f = I G_r^I f .$ It follows that $\sem{G_r^I}$   
mirrors properties of $\sem{G_r}$.  Therefore, as a corollary to Proposition \ref{toiac} and Theorem \ref{ass} we obtain the following result. 

\begin{thm}\label{assprim} For any $r\in [0,1]$, the operator $G_r^I$ is a Feller generator in $C[-1,0]$. Moreover, there are positive constants $K$ and $\omega$ (depending perhaps on $r$) such that 
\begin{equation}\label{assprim:1} 
 \|\e^{tG_r^I} - P_r^I\| \le K\e^{-\omega t}, \qquad t \ge  0, \end{equation}
where $P_r^I= I^{-1} P_r I$ for $P_r$ defined in \eqref{app:2'}, \ie 
\[ P_r^I f (x) = rf(0) + (1-r) \int_{-1}^0 f(y) \ud y, \qquad x \in [-1,0].\]   \end{thm}

\subsection{The vertical component: a generation theorem}\label{tvc:agt} 

\subsubsection{Sticky membrane at $x=0$; no communication between the intervals $[-1,0-]$ and $[0+,1]$} Let $C(U)$ be the space of continuous functions on the union $U$ of two unit intervals $U:= [-1,0-]\cup [0+,1]$. Here, similarly as before, we imagine that there is an infinitely thin membrane at $x=0$ and think of $0-$ and $0+$ as the points to the immediate left and to the immediate right of this membrane, respectively.

Each element $f$ of $C(U)$ may be thought of as the sum of $f_1\in C[-1,0]$ and $f_2 \in C[0,1]$ defined by 
\[ f_1 \coloneqq f_{|[-1,0-]} \mquad { and }  f_2 \coloneqq  f_{|[0+,1]}.\]
This is to say that $C(U)$ is a direct product of its two subspaces which may be identified with $C[-1,0]$ and $C[0,1]$. 

With this convention, given $p,q\in [0,1]$, it makes sense to define
\[ T(t) f = \e^{tG_p^I} f_1 + \e^{tG_q} f_2 \]
where $G_p^I$ and $G_q$ are defined in Section \ref{abb}. It is clear that $\left ( T(t) \right )_{t\ge 0}$ is a conservative Feller semigroup in $C(U)$ and that its generator, say $A_0$, is defined on the domain composed of $f$ such that $f_1 \in \dom{G_p^I}$ and $f_2 \in \dom{G_q}$ by the formula 
\[ A_0 f = G_p^I f_1 + G_q f_2 = f_1'' +f_2''.\]
In other words, an $f$ belongs to $\dom{A_0}$ iff when restricted to either of the two subintervals forming $U$ it is twice continuously differentiable and the following boundary conditions are satisfied:
\begin{align*} f'(-1)=f'(1)&=0,  \\ p f''(0-) + (1-p) f'(0-) &= 0, \\   q f''(0+) - (1-q) f'(0+) &= 0.\end{align*}
Moreover, $A_0f = f''.$ This operator describes two independent Brownian motions in two non-communicating intervals. In the right interval the related process is a Brownian motion with reflecting barrier at $x=1$ and a sticky barrier at $0+$ with stickiness coefficient $q$. In $[-1,0-]$ the process is a mirror image of an analogous Brownian motion with stickiness coefficient $p$.

As a corollary to Theorems \ref{ass} and \ref{assprim} we also have the following information on the asymptotic behavior of the semigroup/process under consideration.

\begin{thm}\label{assbis} There are positive constants $K$ and $\omega$ (depending perhaps on $p$ and $q$) such that 
\begin{equation}\label{assprim:2} 
 \|\e^{tA_0} - P_{p,q}\| \le K\e^{-\omega t}, \qquad t \ge  0, \end{equation}
where $\e^{tA_0}=T(t)$ and $ P_{p,q}$ is given by $ P_{p,q} f= (P_p^I f_1,P_q f_2).$ \end{thm}

\subsubsection{A Greiner-like perturbation of the generator leading to communication between the intervals}  
Let $A$ in $C(U)$ be defined as follows. Its domain is composed of $f\in C(U)$ such that $f_1\in C^2[-1,0]$, $f_2 \in C^2[0,1]$ and $f'(1) = f'(-1)=0$. Also,  
\[ Af = f_1'' +f_2''= f''.\]
Next, let $L:\dom{A}\to \R^2$ be defined by 
\[ Lf = (p f''(0-) + (1-p) f'(0-), \\   q f''(0+) - (1-q) f'(0+) ),\]
so that, in particular, we see that $A_0$ of the previous subsection is $A$ restricted to $\ker L$.

Let $\alpha$ and $\beta$ be non-negative  numbers, let $\mu$ be a Borel probability measure on $[-1,0]$ and let $\nu$ be a Borel probability measure on $[0,1]$.  Given such data, we define $\Phi: C(U)\to \R^2$ by 
\begin{equation}\label{maingen:0} \Phi f = \left ( \alpha [\nu (f) - f(0-)], \beta [ \mu (f)-f(0+)] \right ), \qquad f \in C(U).\end{equation}Here and in what follows, for $f\in C(U)$, we write $\mu (f)$ and $\nu (f)$ to denote $\int_{-1}^0 f_1 \ud \mu$ and $\int_0^1 f_2\ud \nu$, respectively. Our goal in this subsection is to show that the operator \[ A_\Phi\coloneqq A_{|\dom{A_\Phi}},\] \ie the operator $A$ restricted to 
\[ \dom{A_\Phi}\coloneqq \{f \in \dom{A}; Lf = \Phi f\}\]
is a conservative Feller generator. 

In the related process, communication between the intervals $[-1,0-]$ and $[0+,1]$ is possible through a semi-permeable membrane at $x=0$. A~particle starting in $[0+,1]$ performs a sticky Brownian motion in this interval (with reflecting barrier at $x=1$), but the time it spends at the sticky boundary $x=0^+$ is measured (for $r=0$ this measuring is done by the L\'evy local time for Brownian motion, \cites{ito,liggett}). After a sufficiently long random time is spent at the boundary, the particle filters through the membrane to its other side. The larger is $\beta$ the shorter is the time needed to filter through the membrane, and thus it is appropriate to refer to $\beta$ as the permeability coefficient. In particular, for $r=1$, the time spent at the boundary is exponential with parameter $\beta$; in agreement with \cite{fellera4} p. 3, this process will be called an elementary jump process. For $r=0$ the time spent at the boundary is exponential (with the same parameter) with respect to the L\'evy local time; such processes are termed snapping out Brownian motions  by Lejay \cite{lejayn}.  
The measure $\mu$ describes the particle's position after it filters through the membrane. Intuitively, the most natural choice seems to be $\mu (f) = f(0-)$ (the Dirac measure concentrated at $x=0-$), describing the situation where the particle after filtering through the membrane starts its motion from the closest vicinity of $x=0$.   However, for $p=q=1$ (\ie in the case of elementary jump through the membrane) this leads to a rather uninteresting dynamics and thus we decided to work with a more general $\mu$. Needless to say, in $[-1,0-]$ the process is a mirror reflection of a similar Brownian motion with stickiness coefficient $p$ and permeability coefficient $\alpha$. 

We note that $A_0$ and $A_\Phi$ are restrictions of the same operator to different domains. Hence, in proving the following generation theorem we use the seminal ideas of Greiner \cite{greiner}, who pioneered the research on domain changing perturbations of semigroups' generators. (See also \cite{nickelnewlook} for an interesting perspective on Greiner's result.)


\begin{thm} \label{maingen} The operator $A_\Phi$ is a conservative Feller generator. 
\end{thm}
\begin{proof}
As in \cite{knigaz} p. 17 it can be shown that $A_\Phi$ satisfies the positive maximum principle. Moreover, it is clear that $1_U \in \dom{A_\Phi}$ with $A_\Phi 1_U = 0$ (where, of course, $1_U(x) =1 $ for $x \in U$). Since $A_\Phi$ is also densely defined, we will be done once existence of an $f\in \dom{A_\Phi}$ satisfying 
\begin{equation}\label{maingen:1} \lam f - A_\Phi f = g \end{equation}
is established for a fixed $\lam >0 $ and  all $g\in C(U)$.

This is where we follow the approach of Greiner. First we note that the kernel of $\lam - A$ of is spanned by $k_{1,\lam}\in C[-1,0]$ and $k_{2,\lam}\in C[0,1]$ defined by 
\begin{align*}
k_{1,\lam} (x) &= \cosh \sqrt \lam (x+1), \qquad x \in [-1,0],\\
k_{2,\lam} (x) &= \cosh \sqrt \lam (x-1), \qquad x \in [0,1].\end{align*}
For an $f = Ck_{1,\lam} +Dk_{2,\lam} $ in this kernel ($C$ and $D$ are real constants),
\[ Lf = (Cm_\lam (p),Dm_\lam (q))\] 
where 
\[ m_\lam (r) = r \lam \cosh \sqrt \lam + (1-r) \sqrt \lam \sinh \sqrt \lam .\]
Thus, $L$ establishes a one-to-one correspondence between $\ker (\lam - A)$ and $\R^2$ with $L_\lam \coloneqq (L_{\ker (\lam - A)})^{-1}, L_\lam : \R^2\to \ker (\lam - A)$ given by 
\begin{equation}\label{maingen:1-a} L_\lam (x_1,x_2) = \frac {x_1}{m_\lam(p)} k_{1,\lam} +  \frac {x_2}{m_\lam(q)} k_{2,\lam}. \end{equation}

We note that 
 \begin{equation}\label{maingen:1a} \|L_\lam\|\le \max_{r=p,q}  \frac {\cosh \sqrt \lam}{m_\lam (r)}\end{equation} (here, $\R^2$ is equipped with the max norm). Also,
\begin{equation} \label{maingen:1b}\frac {\cosh \sqrt \lam}{m_\lam (r)}\le \begin{cases} \frac 1{r\lam}, & r >0 ,\lam >0,\\ \frac{\cosh \sqrt \lam}{\sqrt \lam  \sinh \sqrt \lam} \le \frac {M_0}{\sqrt \lam}, &  r=0,\lam >1,\end{cases}
\end{equation}
where $M_0 \coloneqq \sup_{x\ge 1} \frac{\cosh x}{\sinh x}$ is finite, since the function involved here is continuous and has finite limits at $x=1$ and $x=\infty$. It follows that for sufficiently large $\lam$ the map $L_\lam \Phi$ has norm smaller than $1$ and thus $I_{C(U)}-L_\lam \Phi$ is invertible. 

Consider such a $\lam $ and a $g\in C(U)$. Let \[ f\coloneqq (I_{C(U)}-L_\lam \Phi)^{-1} \rez{A_0} g\] so that $f = L_\lam \Phi f + \rez{A_0}g $ (comp. \cite{greiner}*{Lemma 1.4}). Since $L_\lam \Phi f\in \ker (\lam - A)\subset \dom{A}$ and $\rez{A_0}g\in \dom{A}$, we see that $f\in \dom{A}$. Then, the calculation $Lf = LL_\lam \Phi f + L\rez{A_0} g = \Phi f + 0 = \Phi f,$ shows that $f $ belongs to $\dom{A_\Phi}$. Moreover, since $A\rez{A_0}g = A_0 \rez{A_0} g = \lam \rez{A_0}g - g$ and $L_\lam \Phi f$ belongs to $\ker (\lam - A)$, 
\begin{align*} A_\Phi f = A f& = A(L_\lam \Phi f + \rez{A_0} g) = \lam L_\lam \Phi f +   \lam \rez{A_0}g - g \\
&= \lam f - g, \end{align*} 
proving that $f$ solves the resolvent equation \eqref{maingen:1}.  \end{proof}

We note that, since operators satisfying the positive maximum principle are dissipative (see e.g. \cite{ethier}*{Lemma 2.1, p. 165}), the solution to the resolvent equation is unique. Hence, 
as a by-product of the proof, we obtain 
\begin{equation} \rez{A_\Phi} =  (I_{C(U)}-L_\lam \Phi)^{-1}  \rez{A_0}, \label{maingen:2}\end{equation}
for all $\lam >0$ such that $I_{C(U)}- L_\lam \Phi $ is invertible (this is in fact repeating Lemma 1.4 in \cite{greiner} in the context of Feller generators).

\subsection{The vertical component: a limit theorem} \label{tvc} Before continuing, we recall that the classical Trotter--Kato Theorem (see e.g. \cites{goldstein,pazy}) says that strongly continuous equibounded semigroups \( \sem{A^\eps}, \eps \in (0,1]\) in a Banach space $E$ converge as $\eps \to 0$ to a strongly continuous semigroup $\sem{B}$, \ie 
\begin{equation}\label{alt:1} \grae \e^{tA^\eps } f= \e^{tB}f, \qquad t \ge 0, f \in E, \end{equation}
iff 
\[ \grae \rez{A^\eps} f = \rez{B} f, \qquad  f \in E, \]
for some/all $\lam >0$; moreover, then the limit \eqref{alt:1} is uniform in $t$ in compact subsets of $[0,\infty)$. 
In other words, such regular convergence of semigroups is completely characterized (see also \cites{abhn,kniga,knigaz,ethier} for the Sova--Kurtz version \cites{kurtz,sova} of this characterization). 

However, in the theory of singular perturbations and in the particular example we are studying here the limit semigroup is strongly continuous only on a subspace of $E$: we are facing a limit theorem of the form 
\begin{equation}\label{alt:2}  \grae \e^{tA^\eps} f = \e^{tB} P f , \qquad t >0, f \in E\end{equation}
where \sem{B} is a strongly continuous semigroup on a subspace $E_0$ of $E$ and $P$ is a projection on $E_0$ (in the sense that $P^2=P$ and $Pf =f , f \in E_0$). Needless to say, in this case the classical theory does not work and, in particular, condition 
\begin{equation}\label{alt:3} \grae  \rez{A^\eps} f = \rez{B} Pf, \qquad  f \in E, \end{equation}
for all (some) $\lam >0$ is necessary but not sufficient for \eqref{alt:2} (see \cite{deg} or \cite{knigaz}).

As we have see in Section \ref{ldwa},  \eqref{alt:3} may imply \eqref{alt:2} provided that the semigroups involved possess additional regularity properties (like, for example, uniform holomorphicity -- see e.g. \cite{knigaz}, Chapters 31 and 41 for details). A different set of conditions guaranteeing that \eqref{alt:3} implies \eqref{alt:2} has been given by T. G. Kurtz \cites{ethier,kurtzper,kurtzapp}. While Kurtz's singular convergence theorem is usually expressed in terms of the so-called extended limit of generators, for our subsequent analysis the following resolvent-version will be more practical. This result may be easily deduced e.g. from combined Lemma 7.1 and Theorem 42.2 in \cite{knigaz}.    

\begin{thm} \label{tkurtza} Suppose $A^\eps, \eps \in (0,1]$ are generators of strongly continuous equibounded semigroups. Suppose also that  for some $\lam >0$
\[ \grae \rez{\eps^2 A^\eps} = \rez{A_0} \]
where $A_0$ is the generator of a strongly continuous semigroup \sem{A_0} such that 
\[ Pf\coloneqq \grat \e^{tA_0} f , \qquad f \in E\]
 exists. Then condition \eqref{alt:3}  (for some $\lam >0$, with the same $P$) implies \eqref{alt:2}, and the limit is uniform in $t$ in compact subsets of $(0,\infty)$; for $f\in E_0$ the limit is uniform in $t$ in compact subsets of $[0,\infty).$ 
\end{thm}

We will apply this theorem to the Feller generators \begin{equation}\label{juzsamniewiem} A^\eps \coloneqq \eps^{-2} A_{\eps^2 \Phi}. \end{equation}
In other words, we will study the situation in which diffusion is very fast while permeability coefficients of the membrane are low.

\newcommand{\efel}{f^-}
\newcommand{\efer}{f^+}

Let $E_0\subset C(U)$ be the subspace composed of functions which are constant in each of the subintervals forming $U$ (\ie for $f\in E_0$ both $f_1$ and $f_2$ are constant functions). Each member of $E_0$ may be naturally identified with two real numbers, say $\efel$ and $\efer$, and $E_0$ may be naturally identified with $\R^2$ with maximum norm.

\begin{thm} \label{limit} Let $B$ be the operator in $E_0$ which may be identified with the matrix of \eqref{macierz}.
In other words, $B(\efel,\efer) = (\alpha (\efer-\efel), \beta(\efel-\efer)).$ Then 
\[ \grae \e^{tA^\eps} f = \e^{tB} P_{p,q} f , \qquad t >0, f \in C(U),\]
where $P_{p,q}$ is defined in Theorem \ref{assbis} and the limit is uniform in $t$ in compact subsets of $(0,\infty)$; for $f\in E_0$ the limit is uniform in compact subsets of $[0,\infty)$. 
\end{thm}

For the proof of this result we need the following lemma. 

\begin{lem}\label{lem:limit} For sufficiently large $\lam$, 
\[ \grae \rez{A_\eps}f = \rez{B} P_{p,q}f, \qquad f \in C(U).\]
\end{lem}

\begin{proof} Solving the resolvent equation for $A^\eps $: $\lam f - A^\eps f = g$ is equivalent to solving the resolvent equation for $A_{\eps^2 \Phi} $ with $\lam $ replaced by $\eps^2 \lam$ and $g$ replaced by $\eps^2 g$. On the other hand, by \eqref{maingen:1b}, \[ \frac {\eps^2 \cosh \eps \sqrt \lam}{m_{\eps^2 \lam}  (r)}\le \frac 1{r\lam}, \qquad r \in (0,1], \lam >0, \eps \in (0,1].\] Moreover, 
\begin{align*} \frac {\eps^2 \cosh \eps \sqrt \lam}{m_{\eps^2 \lam}  (0)}&= \frac {\eps \cosh \eps \sqrt \lam }{\sqrt \lam \sinh \eps \sqrt \lam } \le \begin{cases} \frac {\eps M_2}{\sqrt \lam }\le \frac {M_2}{\sqrt \lam}, & \eps \sqrt \lam \ge 1, \\  \frac {M_3} \lam , & \eps \sqrt \lam\in (0,1], \end{cases} \quad \eps \in (0,1],\end{align*}
where $M_2\coloneqq  \sup_{x>0 }\frac {\cosh x}{\sinh x} $ and $M_3 \coloneqq \sup_{x\in (0,1]} \frac {x \cosh x}{\sinh x}$ are finite because the functions $x\mapsto  \frac {\cosh x} {\sinh x}$ 
and $x\mapsto \frac {x \cosh x}{\sinh x}$
are continuous and have finite limits at appropriate intervals' ends.

It follows, by \eqref{maingen:1a}, that for sufficiently large $\lam $ the norm of $\eps^2 L_{\eps^2 \lam}  \Phi$ is smaller than $1$, regardless of the choice of $\eps $, and so $I -  \eps^2  L_{\eps^2 \lam}\Phi$ is invertible for all $\eps \in (0,1)$. Therefore, for such $\lam$, by \eqref{maingen:2}, 
\[ \rez{A^\eps } = \eps^2 (\eps^2 \lam - A_{\eps^2 \Phi})^{-1} = \eps^2 (I_{C(U)} - L_{\eps^2 \lam}\eps^2 \Phi)^{-1} \left ( \eps^2 \lam - A_0 \right )^{-1}.\]

Next, by Theorem \ref{assbis}, $\grae \eps^2 \left ( \eps^2 \lam - A_0 \right )^{-1}= \lam^{-1}  P_{p,q}$, and we are left with analyzing the factor $(I_{C(U)} - L_{\eps^2 \lam}\eps^2 \Phi)^{-1}$.
To this end we observe that for $f \in C(U)$ (see  \eqref{maingen:1-a} and the definition of $\Phi$) 
\[ L_{\eps^2 \lam } \eps^2 \Phi f = \frac {\eps^2 \alpha [\nu (f) - f(0-)]}{m_{\eps^2 \lam} (p)} k_{1,\eps^2 \lam} + \frac {\eps^2 \beta [f(0+)-\mu (f)]}{m_{\eps^2 \lam} (q)} k_{1,\eps^2 \lam} \]
converges, as $\eps \to 0$, to
\begin{align*} \frac {\alpha [\nu (f) - f(0-)]}\lambda 1_{[-1,0]} &+\frac {\beta [\mu (f)-f(0+)]}\lambda 1_{[0,1]} \\ &= \lam^{-1} \left (\alpha [\nu (f) - f(0-)], \beta [ \mu (f)-f(0+)]\right ),\end{align*}
because, as it is easy to check, $\grae \frac{\eps^2}{m_{\eps^2 \lam} (r)} = \lam^{-1}, r \in [0,1].$ Since $\mu$ and $\nu$ are probability measures, for $f=(\efel, \efer)$ in $E_0$, the latter vector is \[\lam^{-1} \left (\alpha (\efer - \efel ),\beta(\efel - \efer) \right )= \lam^{-1} B(\efel, \efer).\]
This shows that \[\grae \rez{A^\eps} f = \left (I_{C(U)} - \lam^{-1} B\right )^{-1} \lam^{-1} P_{p,q}f = \rez{B}P_{p,q}f,\] as claimed.  
\end{proof} 

\begin{proof}[Proof of Theorem \ref{limit}]  A similar (but simpler)  analysis to that presented in Lemma \ref{lem:limit} shows that $\grae \rez{\eps^2 A^\eps} = \grae \rez{A_{\eps^2\Phi}}= \rez{A_0}$ for all $\lam >0$. Since, by Theorem \ref{assbis}, $\grat \e^{tA_0} f =  P_{p,q} f$, Theorem \ref{limit} is a direct consequence of Theorem \ref{tkurtza} and Lemma \ref{lem:limit}.
\end{proof}

\subsection{The semigroups generated by $\onaeps$}\label{dowod}

The space $C(\Omega)$ may be seen as the injective tensor product of the spaces $C(\be)$ and $C(U)$: 
\[ C(\Omega) = C(\be)\tilde \otimes_\epsilon  C(U),\]
see e.g. \cite{ryan}*{pp. 45-50}.
 This means that the supremum norm in $C(U)$ coincides with the injective norm inherited from $C(\be)$ and $C(U)$, and the set of simple tensors, \ie of functions of the form 
\( f\otimes g , f \in C(\be)\times C(U)\) given by $(f\otimes g)(x,y,z)= f(x,y)g(z), (x,y,z)\in \Omega$, is 
linearly dense in $C(\Omega)$. This allows constructing semigroups of operators in $C(\Omega) $ from building blocks available in $C(\be)$ and $C(U)$ (see \cite{nagel}*{pp. 21-24}), as follows. 

Since $\partial \be$ is assumed to be of class $C^{2,\kappa}, \kappa \in (0,1]$, the $2D$ Laplace operator $\ddwad$ with  domain composed of $C^{2,\kappa} (\Omega)$ functions with normal derivatives vanishing on the boundary is closable, and its closure generates a conservative Feller semigroup in $C(\Omega)$ (see \cite{ethier}*{p. 369}). Let $\sem{\ddwadd}$ be this semigroup, and let $\sem{A^\eps}, \eps \in (0,1]$ be the semigroups generated by $A^\eps$ of \eqref{juzsamniewiem}. 

For any $\eps \in (0,1]$ and $t\ge 0$, one may think of the following map defined on the set of simple tensors 
\[ f\otimes g \mapsto (\e^{t \ddwadd} f )\otimes (\e^{tA^\eps} g). \]
Since such tensors form a linearly dense set in $C(\Omega)$, and since the supremum norm in $C(\Omega)$ coincides with the injective tensor norm inherited from $C(U)$ and $C(\be)$ (\cite{ryan}*{pp. 49-50}), this map may be extended to a bounded linear operator, say $\mc T_\eps (t)$, in $C(\Omega)$ with norm $1$. This operator is positive and $\mc T_\eps (t)1_{\Omega} = 1_{\Omega}$. 

In \cite{nagel}*{pp. 21-24} it is shown that so-constructed $\left (\mc T_\eps (t)\right )_{t\ge 0}$ is a strongly continuous semigroup; this semigroup is termed the injective tensor product of semigroups $\sem{\ddwadd}$ and $\sem{A^\eps}$, and denoted 
\[ \mc T_\eps (t) = \e^{t\ddwadd}\tilde \otimes_\epsilon \e^{tA^\eps}.\]
Moreover, the set of linear combinations of simple tensors of the form $f\otimes g, f \in \dom{\ddwadd}, g \in \dom{A^\eps}$ is a core for the generator of this semigroup. 
It is clear that the last statement is also true if instead of $f \in \dom{\ddwadd}$ one considers $f\in \dom{\ddwad}$, and that $\left (\mc T_\eps (t)\right )_{t\ge 0}$ is a conservative Feller semigroup.

\begin{prop}\label{tsgb:prop:1} For any $\eps \in (0,1]$, the operator $\naeps$ of \eqref{cauchy} is closable and its closure generates the semigroup $\left (\mc T_\eps (t)\right )_{t\ge 0}$.   \end{prop}
\begin{proof}Arguing as in \cite{knigaz}*{p. 17} we conclude that at $z=0+$ and $z=0-$, $\partial_z u$ vanishes for $u\in \dom{\naeps}$, and this in turn implies that $\naeps$ satisfies the maximum principle. 

For the sake of this proof, let $\mc D$ be the set of linear combinations of simple tensors of the form $f\otimes g, f \in \dom{\ddwad}, g \in \dom{A^\eps}$, and let $\mc A_\eps $ be the generator of the semigroup $\left (\mc T_\eps (t)\right )_{t\ge 0}$. For a simple tensor $u= f\otimes g \in \mc D$   
\[ \mc A_\eps u = (\ddwad f)\otimes g + f \otimes (A^\eps g)= (\partial_x^2 + \partial_y^2 +  \eps^{-2} \partial_z^2 ) u \]
(see \cite{nagel}*{p. 23}). 
Since it is clear that such a $u$ belongs also to $\dom{\naeps}$, the operators $\naeps$ and $\mc A_\eps $ have the common subdomain $\mc D$ where they coincide.  Next, for any $\lam >0$, the range of $(\lam - \mc A_\eps)_{|\mc D}$ is dense in $C(\Omega)$, because $\mc D$ is a core for $\mc A_\eps$ (see \cite{ethier}*{Proposition 3.1, p. 17}). Therefore, also the range of $\lam - \naeps$ is dense in $C(\Omega)$, since it contains $(\lam - \naeps)_{|\mc D}=(\lam - \mc A_\eps)_{|\mc D}$. Thus, the operator $\naeps $, being clearly densely defined, is closable and generates a conservative Feller generator by Theorem 2.2. in \cite{ethier}*{p. 165}. Also, by the other implication in \cite{ethier}*{Proposition 3.1, p. 17} just alluded to, $\mc D$ is a core for $\naeps$. Hence, $\mc D$ being a common core for $\mc A_\eps $ and $\onaeps$, these two generators must coincide. 
 \end{proof}
 
The subspace $\notzee$ of Section \ref{tmrr} may be considered as an injective tensor product, too. Namely,
\[ \notzee = C(\be ) \tilde \otimes_\epsilon C(\{0-\}\cup \{0+\}). \]
where $C(\{0-\}\cup \{0+\})$, the space of continuous functions on the set $\{0-\}\cup \{0+\}$ with discrete topology may be identified with $\R^2$ with the maximum norm. Therefore, one may think of the injective tensor product semigroup $\left (\mc S (t)\right )_{t\ge 0}$, where 
\[ \mc S(t) := \e^{t\ddwadd} \tilde \otimes_\epsilon \e^{tB} \]
where $B$ was defined in Theorem \ref{limit}.  

\begin{prop} The operator $\mathfrak B$ of Section \ref{tmrr} is the generator of the injective product semigroup  $\left (\mc S (t)\right )_{t\ge 0}$.
\end{prop}

\begin{proof} It will be convenient to identify elements $g\in C(\{0-\}\cup \{0+\})$ with pairs of real numbers written as $\binom{g^-}{g^+}$. With this identification, a member $u$ of $\notzee$ has the form
\begin{equation}\label{propek:1} u= u^- \otimes \binom 10 + u^+ \otimes \binom 01 , \end{equation}
where $u^-$ and $u^+$ are defined in \eqref{tmrr:1}. 

Let, for the sake of this proof, $\mc B$ be the generator of $\left (\mc S (t)\right )_{t\ge 0}$. If $u$ is a member of $\dom{\mathfrak B}$, \ie if $u^-$ and $u^+$ belong to $\dom{\ddwadd}$ then  (by the already cited result from p. 23 in \cite{nagel}) \eqref{propek:1} shows that $u\in \dom{\mc B}$, and 
\begin{align*}
\mc Bu &= \ddwadd u^- \otimes \binom 10 + u^- \otimes B \binom 10
+ \ddwadd u^+ \otimes \binom 01 + u^+ \otimes B\binom 01\\
&= \ddwadd u^- \otimes \binom 10 + u^- \otimes \binom {-\alpha}\beta
+ \ddwadd u^+ \otimes \binom 01 + u^+ \otimes \binom {\alpha}{-\beta}\\
&= (\ddwadd u^- - \alpha u^- + \alpha u^+)\otimes \binom 10 +  (\ddwadd u^+ +\beta u^- - \beta u^+)\otimes \binom 01 \\&= \mathfrak Bu .
\end{align*}
It follows that $\mc B$ extends $\mathfrak B$. However, since both $\mc B$ and $\mathfrak B$ are generators, $\mc B$ cannot be a proper extension of $\mathfrak B$ and we conclude that $\mathfrak B= \mc B.$ \end{proof}

 \begin{proof}[Proof of Theorem \ref{mainth}]
 Since simple tensors form a linearly dense subset of $C(\Omega)$ it suffices to show \eqref{mainth:1} for $u=f\otimes g$ where $f\in C(\be)$ and $g \in C(U).$ By Theorem \ref{limit},
\[ \grae \e^{t\onaeps } (f\otimes g) = (\e^{t\ddwadd} f )\otimes (\grae \e^{tA^\eps}g ) =  (\e^{t\ddwadd} f)\otimes (\e^{tB} P_{p,q} g), \quad t >0.\]
 Since, as a direct calculation shows, $\mc P_{p,q} (f\otimes g) = f \otimes P_{p,q} g$, we have, on the other hand, 
 \[ \e^{t\mathfrak B} \mc P_{p,q} (f\otimes g) =(\e^{t\ddwadd} f)\otimes (\e^{tB} P_{p,q} g),\]
and this completes the proof. \end{proof}


\iftoggle{blind}{}{\textbf {Acknowledgment.}  This research is supported by National Science Center (Poland) grant
2017/25/B/ST1/01804.}


\bibliographystyle{plain}
\bibliography{../../../../bibliografia}
\end{document}